\documentclass[11pt,reqno]{amsart}
\usepackage{mathrsfs}
\usepackage{amscd}
\usepackage{fullpage}
\usepackage[headheight=110pt,top=1in, bottom=.9in, left=0.8in, right=0.8in]{geometry}

\usepackage{graphicx}
\usepackage{inputenc}

\usepackage{url}
\usepackage{amssymb}
\usepackage{amsmath}

\allowdisplaybreaks

\newcommand{\pFq}[5]{\ensuremath{{}_{#1}F_{#2} \left( \genfrac{}{}{0pt}{}{#3}
{#4} \bigg| {#5} \right)}}

\renewcommand{\Re}{\operatorname{Re}}

\newcommand{\G}{\Gamma}

\renewcommand{\(}{\left\(}
\renewcommand{\)}{\right\)}
\renewcommand{\[}{\left\[}
\renewcommand{\]}{\right\]}
\newcommand{\Z}{\mathbb{Z}}
\newcommand{\K}{\mathbb{K}}
\newcommand{\Q}{\mathbb{Q}}
\newcommand{\C}{\mathbb{C}}
\newcommand{\R}{\mathbb{R}}
\newcommand{\N}{\mathbb{N}}

\let\dotlessi=\i
\renewcommand{\i}{\infty}
\numberwithin{equation}{section}
\theoremstyle{plain}
\newtheorem{theorem}{Theorem}[section]
\newtheorem{lemma}[theorem]{Lemma}
\newtheorem{conjecture}[theorem]{Conjecture}
\newtheorem{corollary}[theorem]{Corollary}
\newtheorem{proposition}[theorem]{Proposition}

\theoremstyle{remark}
\newtheorem*{remark}{Remark}

\usepackage{bigints}
\usepackage{suffix}
\usepackage{mathtools}
\DeclarePairedDelimiterX\MeijerM[3]{\lparen}{\rparen}%
{\begin{smallmatrix}#1 \\ #2\end{smallmatrix}\delimsize\vert\,#3}

\newcommand\MeijerG[8][]{%
  G^{\,#2,#3}_{#4,#5}\MeijerM[#1]{#6}{#7}{#8}}

\WithSuffix\newcommand\MeijerG*[7]{%
  G^{\,#1,#2}_{#3,#4}\MeijerM*{#5}{#6}{#7}}

\numberwithin{theorem}{section}
\numberwithin{equation}{section}

\DeclareUnicodeCharacter{2212}{-}

\begin{document}
\title[zeta values over imaginary quadratic field]{Explicit identities on zeta values over imaginary quadratic field}
\author{Soumyarup Banerjee and Rahul Kumar}
\address{ Discipline of Mathematics, Indian Institute of Technology Gandhinagar, Palaj, Gandhinagar 382355, Gujarat, India}\email{soumyarup.b@iitgn.ac.in;  rahul.kumr@iitgn.ac.in}

\begin{abstract}
In this article, we study special values of the Dedekind zeta function over an imaginary quadratic field. The values of the Dedekind zeta function at any even integer over any totally real number field is quite well known in literature. In fact, in one of the famous article, Zagier obtained an explicit formula for Dedekind zeta function at point 2 and conjectured an identity at any even values over any number field. We here exhibit the identities for both even and odd values of the Dedekind zeta function over an imaginary quadratic field which are analogous to Ramanujan's identities for even and odd zeta values over $\Q$. Moreover, any complex zeta values over imaginary quadratic field may also be evaluated from our identities. 
\end{abstract}
\maketitle
\vspace{-0.8cm}
\tableofcontents

\section{Introduction}\label{intro}
We begin with the famous quote by Zagier \cite{Zagier2} about zeta function that ``{\it Zeta functions of various sorts are all-pervasive objects in modern number theory, and an ever-recurring theme is the role played by their special values at integral arguments, which are linked in mysterious ways to the underlying geometry and often seem to dictate the most important properties of the objects to which the zeta functions are associated}." In the literature, the special values of the Riemann zeta function are well studied. The zeta values at even integers were established by Euler in 1735 which precisely states that for all $m \in \N$, we have
\begin{equation}\label{Zeta at even integer}
\zeta(2m) = (-1)^{m+1} \frac{(2\pi)^{2m}B_{2m}}{2(2m)!}
\end{equation}
where $B_{2m}$ denotes the $2m$-th Bernoulli numbers. More surprisingly, the value of Riemann zeta function at odd integer is still mysterious, even the question whether the zeta values at odd integers are rational or irrational, is solved only for the value $\zeta(3)$ by Apery \cite{Apery}. Zudilin \cite{Zudilin} has shown that atleast one of the four members $\zeta(5)$, $\zeta(7)$, $\zeta(9)$, $\zeta(11)$ is irrational. A celebrated identity due to Ramanujan for odd zeta values as \cite[pp. 319-320, formula (28)]{Ramanujan}, states that for any $\alpha$, $\beta > 0$ with $\alpha\beta = \pi^2$, we have  
\begin{multline}\label{Ramanujan formula}
\alpha^{-m}\left\lbrace \frac{1}{2}\zeta(2m+1) + \sum_{n=1}^\infty\frac{n^{-2m-1}}{e^{2n\alpha} - 1} \right\rbrace
= (-\beta)^{-m}\left\lbrace \frac{1}{2}\zeta(2m+1) + \sum_{n=1}^\infty\frac{n^{-2m-1}}{e^{2n\beta} - 1} \right\rbrace\\
-2^{2m} \sum_{k=0}^{\infty} \frac{(-1)^k B_{2k} B_{2m+2-2k}}{(2k)!(2m+2-2k)!}\alpha^{m+1-k}\beta^k.
\end{multline}
The zeta values over a number field have also been studied extensively. The result \eqref{Zeta at even integer} of Euler was generalized further for any totally real number field by Klingen \cite{Klingen} and Siegel \cite{Siegel}, who precisely showed that for any totally real number field $\K$ of degree $n$ with discriminant $D$,  
\begin{equation*}
\zeta_\K(2m) = \frac{q_m\pi^{2mn}}{\sqrt{D}}\qquad (m \in \N ),
\end{equation*} 
where $q_m$ is some fixed non-zero rational number. In particular, for a real quadratic field $\K$ one can obtain a more precise evaluation [cf. \cite{BCH}] such as
\begin{equation*}
\zeta_\K(2m) = \frac{\tau(\chi_{D})(2\pi)^{4m}B_{2m}B_{2m, \overline{\chi}_D}}{4\left((2m)! \right)^2 D^{2m}}
\end{equation*}
where $\tau(\chi_{D})$ is the Gauss sum associated to $\chi_D$ and $B_{2m, \chi_D}$ is the $2m$-th generalized Bernoulli number associated to $\chi_D$. Zagier \cite{Zagier} obtained an explicit formula for $\zeta_\K(2)$ over any number field $\K$, which precisely states that for any number field $\K$ with discriminant $D$ and signature $(r_1, r_2)$, the following finite sum
\begin{equation}\label{Eqn: Zagier}
\zeta_\K(2) = \frac{\pi^{2r_1+2r_2}}{\sqrt{D_\K}}\sum_{\nu}c_{\nu} A(x_{\nu, 1}) \cdots A(x_{\nu, r_2})
\end{equation}
holds, where $A(x)$ is the real-valued function given by the following integral
$$A(x):= \int_0^x \frac{1}{1+t^2}\log \frac{4}{1+t^2} {\rm d}t,$$ 
$c_\nu$ are rational, and $x_{\nu,j}$ are real algebraic numbers. 

Recently, Dixit et al. \cite{DKK} studied the series $\sum_{n=1}^\infty \sigma_a(n) e^{-ny}$ associated to the divisor function $\sigma_a(n)$ and obtained an explicit transformation of this series for any complex number $a$. As a special case, the result provides the transformation formulas for Eisenstein series, Eichler integrals, Dedekind eta function and Ramanujan's formula \eqref{Ramanujan formula}. On the other hand, for $a$ even new transformation formulas have been obtained in \cite[Theorem 2.11, Corollary 2.13]{DKK}. 

In this article, we investigate zeta values over an imaginary quadratic field through a series which is analogous to $\sum_{n=1}^\infty \sigma_{ a}(n) e^{-ny}$. Throughout the paper, we let our imaginary quadratic field be $\K$ with discriminant $D$  ( absolute value $D_\K$ ), class number $h$ and number of roots of unity to be $w$. Let $\mathcal{O}_\K$ be its ring of integers and $v_{\K}(m)$ denote the number of non-zero integral ideals in $\mathcal{O}_\K$ with norm $m$. Let $\mathfrak{N}$ be the norm map of $\K$ over $\Q$ and $\mathfrak{N}_{\K/\Q}(I)$ denotes the absolute norm of any non-zero integral ideal $I \subseteq \mathcal{O}_\K$.  We denote the Dedekind zeta function over any imaginary quadratic field $\K$ by $\zeta_\K(s)$ and the $L$-function  associated to the quadratic character $\chi_D = \left(\frac{D}{\cdot}\right)$ by $L(s, \chi_D)$ where $(\frac{\cdot}{\cdot})$ denotes the Jacobi symbol. We define the general divisor function over $\K$ by
\begin{equation}\label{Divisor function}
\sigma_{\K, a}(n) := \sum_{\substack{I \subseteq \mathcal{O}_\K \\  \mathfrak{N}_{\K/\Q}(I) \mid n}} \left( \mathfrak{N}_{\K/\Q}(I)\right)^a = \sum_{d\mid n} v_{\K}(d) d^a
\end{equation}
where $a$ is any complex number.

For an imaginary quadratic field $\K$, it immediately follows from \eqref{Eqn: Zagier} that $\zeta_\K(2)$  can be expressed by the finite sum
\begin{equation}\label{Value of zeta_K(2)}
\zeta_\K (2) = \frac{\pi^2}{\sqrt{D_\K}}\sum_\nu c_\nu A(x_\nu). 
\end{equation}
We obtain an alternate expression for $\zeta_\K(2)$ over an imaginary quadratic field in the following theorem.
\begin{theorem}\label{a=-1 case}
Let $\Re(y)>0$. Then, we have
\begin{multline*}
\zeta_\K(2)=\frac{y}{2}\left\{L'(1,\chi_D)+L(1,\chi_D)\left(2\gamma-\log\left(\frac{2\pi}{y}\right)+\frac{y\sqrt{D_\K}}{12\pi}\right)\right\}+y\sum_{n=1}^\infty\sigma_{\K,-1}(n)e^{-ny}\\
+\frac{4y}{\sqrt{D_\K}}\sum_{n=1}^\infty\frac{\sigma_{\K,1}(n)}{n}\mathrm{kei}\left(4\pi\sqrt{\frac{2n\pi}{yD_\K}}\right),
\end{multline*}
where the function  ${\rm kei}(x)$ is the Kelvin function defined in \S \ref{Prilim}.
\end{theorem}

In the same article \cite{Zagier},  Zagier conjectured an explicit identity to evaluate the Dedekind zeta function at any even integer over any number field.
\begin{conjecture}[Zagier]
For each $m\in \N$, let $A_m(x)$ be the real valued function 
$$A_m(x) = \frac{2^{2m-1}}{(2m-1)!}\int_0^\infty \frac{t^{2m-1}}{x \sinh^2(t)+x^{-1}\cosh^2(t)} {\rm d}t.$$
Then the value of $\zeta_\K(2m)$ for an arbitrary number field $\K$ with signature $(r_1, r_2)$ and discriminant $D$ may equal $\pi^{2m(r_1+r_2)}/\sqrt{|D|}$ times a rational linear combination of products of $r_2$ values of $A_m(x)$ at algebraic arguments.
\end{conjecture} 

We provide an explicit expression for $\zeta_\K(s)$  at even and odd arguments in the following theorem which may be considered as an analogue of Ramanujan's formula \eqref{Ramanujan formula} over an imaginary quadratic field. 
\begin{theorem}\label{Ramanujan analogue}
For any natural number $m$ and any complex number $\alpha, \beta$  with $\Re(\alpha), \Re(\beta) >0$ and  $\alpha \beta = \frac{D_\K^2}{16\pi^2}$, we have
\begin{align}\label{a=-2m-1 case}
& \alpha^{-m}\left\{\frac{1}{2}\zeta_{\mathbb{K}}(2m+1)+\sum_{n=1}^\infty\sigma_{\mathbb{K},-2m-1}(n)e^{-A\pi n\alpha}-\frac{1}{A\pi\alpha}\zeta_{\mathbb{K}}(2m+2)\right\}\nonumber\\
& =(-\beta)^{-m}\left\{\frac{\pi h}{w\sqrt{D_\K}}\zeta(2m+1)-\frac{4}{\sqrt{D_\K}}\sum_{n=1}^\infty \frac{\sigma_{\mathbb{K}, 2m+1}(n)}{n^{2m+1}} {\rm kei} \left(A\sqrt{\pi n \beta} \right)-\frac{2h}{wA\pi \beta}\zeta(2m+2)\right\}\nonumber\\
&\hspace{5cm}+\frac{2^{3m-1}}{D_\K^{m-2}\pi^4}\sum_{k=1}^{m}(-1)^{m-k}\zeta(2m+2-2k)\zeta_{\mathbb{K}}(2k) \alpha^{m+3-k}\beta^{m-1+k}
\end{align}
and
\begin{align}\label{a=-2m case}
&\alpha^{-(m-\frac{1}{2})}\left \{\frac{1}{2}\zeta_{\mathbb{K}}(2m) + \sum_{n=1}^\infty\sigma_{\mathbb{K},-2m}(n)e^{-A\pi n\alpha} -\frac{1}{A\pi\alpha}\zeta_{\mathbb{K}}(2m+1)\right\} = (-1)^{m+1} \beta^{-(m-\frac{1}{2})} \nonumber\\
& \times\left\{ \frac{1}{\pi}\zeta(2m)\left(\gamma+\log\left(\frac{A\beta}{2}\right) L(1, \chi_D) + L'(1, \chi_D) \right)\right.
-\frac{4\pi^{2-2m}}{\sqrt{D_\K}}\sum_{n=1}^\infty \frac{\sigma_{\mathbb{K}, 2m}(n)}{n^{2m}} {\rm ker} \left(A\sqrt{\pi n \beta} \right)\nonumber\\
&\hspace{.4cm}\left.-\frac{2h}{w\sqrt{D_\K}}\zeta'(2m) \right\} + \pi^{2m-3} \sum_{k=1}^{m-1} (-1)^{m-1-k} (2\pi)^{2m-2k} \zeta(2m-2k) \zeta_{\K} (2k+1) \alpha^{1-k-m}\beta^{k-2m+\frac{3}{2}}
\end{align}
where $A =\frac{8\pi}{D_\K}$ and the functions ${\rm ker}(x)$, ${\rm kei}(x)$ are the Kelvin functions which are defined in \S \ref{Prilim}.
\end{theorem}

\begin{remark}
We are not claiming here that the above theorem solves the conjecture over an imaginary quadratic field  but certainly it provides an alternate expression for $\zeta_\K(2m)$ over any imaginary quadratic field. An analogue to Lerch's result \cite{lerch} over an imaginary quadratic field can be obtained from \eqref{a=-2m-1 case} by substituting $m$ by $2m+1$.
\end{remark}
The following corollary provides a  representation for $\zeta_\K(3)$ in terms of $\zeta_\K(2)$. The latter is well-known due to Zagier's identity \eqref{Value of zeta_K(2)}.
\begin{corollary}\label{zetaK3}
We have
\begin{multline*}
\zeta_\K(3) = 2\pi \left \{ \sum_{n=1}^\infty \sigma_{\K, -2}(n) e^{-2n\pi} + \frac{4}{\sqrt{D_\K}} \sum_{n=1}^\infty \frac{\sigma_{\K, 2}(n)}{n^2} {\rm ker} \left(4\pi \sqrt{\frac{n}{D_\K}}\right) \right\} - \frac{\pi^3}{3}(\gamma + L'(1, \chi_D))\\
+ \frac{4\pi h}{w\sqrt{D_\K}}\zeta'(2)  + \pi \zeta_\K(2).
\end{multline*}
\end{corollary}
The above corollary follows immediately by letting  $m = 1$ and $\alpha = \beta = \frac{D_\K}{4\pi}$ in \eqref{a=-2m case}. It is natural to ask whether it is possible to find an explicit identity for Dedekind zeta function over any imaginary quadratic field $\K$ at complex arguments. The next theorem answers the question.
\begin{theorem}\label{Lambert series}
For $\mathrm{Re}(y)>0$ and $\mathrm{Re}(a)>-1$, the identity 
\begin{align}\label{Eqn:Lambert series}
\sum_{n=1}^\infty &\sigma_{\mathbb{K},a}(n)e^{-ny}+\frac{1}{2}\zeta_{\mathbb{K}}(-a)-\frac{\zeta_{\mathbb{K}}(1-a)}{y}-\frac{L(1, \chi_D)\Gamma(a+1)\zeta(a+1)}{y^{a+1}}\nonumber\\
&=\frac{4\pi^{2-2a}D_\K^{a-\frac{1}{2}}}{y\sin(\pi a)} \sum_{n=1}^\infty \sigma_{\mathbb{K},-a}(n)\Bigg\{\frac{2^{-2a}}{\Gamma^2(1-a)}\pFq14{1}{1-\frac{a}{2},1-\frac{a}{2},\frac{1-a}{2},\frac{1-a}{2}}{-\frac{4\pi^6n^2}{y^2D_\K^2}}
\nonumber\\
&\hspace{2cm}-\left( \frac{4\pi^6n^2}{y^2D_\K^2} \right)^{\frac{a}{2}}\bigg(\cos\left(\frac{\pi a}{2}\right) {\rm ber}\left(4\pi\sqrt{\frac{2n\pi}{yD_{\mathbb{K}}}} \right)-\sin\left(\frac{\pi a}{2}\right){\rm bei}\left(4\pi\sqrt{\frac{2n\pi}{yD_{\mathbb{K}}}} \right)\bigg)\Bigg\}
\end{align}
holds, where the functions ${\rm ber}(x)$, ${\rm bei}(x)$ are the Kelvin functions
and ${}_pF_q$ denotes the hypergeometric function which are defined in \S \ref{Prilim}.
\end{theorem}
\begin{remark}
An analogous version of the above theorem over $\Q$ was obtained in \cite{DKK}.
\end{remark}
We next abbreviate $\sigma_{\K, 0}$ by $\sigma_{\K}$ and obtain the following important corollary from the above theorem by substituting $a=0$. 
\begin{corollary}\label{a=0 case}
Let $\gamma$ be Euler's constant. 
For $\mathrm{Re}(y)>0$ and $\mathrm{Re}(a)>-1$, we have
\begin{equation*}\label{a=0 case eqn}
\sum_{n=1}^\infty \sigma_{\mathbb{K}}(n)e^{-ny} - \frac{h}{2w}-\frac{L'(1, \chi_D)+L(1, \chi_D) (\gamma - \log(y))}{y} = \frac{8\pi}{y\sqrt{D_{\mathbb{K}}}}\sum_{n=1}^\infty \sigma_{\mathbb{K}}(n)\mathrm{ker}\left(4\pi\sqrt{\frac{2n\pi}{yD_{\mathbb{K}}}} \right).
\end{equation*}
\end{corollary}
Theorem \ref{Lambert series} can be extended in the half-plane $\Re(a) > -2m - 3$, where $m$ is any non-negative integer through analytic continuation. 
\begin{theorem}\label{Analytic continuation}
If $\mathrm{Re}(y)>0$ and  $\mathrm{Re}(a)>-2m-3$ with $m\in\mathbb{N}\cup\{0\}$, then the following identity holds: 
 \begin{align}\label{Eqn:Analytic continuation}
&\sum_{n=1}^\infty \sigma_{\mathbb{K},a}(n)e^{-ny}+\frac{1}{2}\zeta_{\mathbb{K}}(-a)-\frac{\zeta_{\mathbb{K}}(1-a)}{y} -\frac{2\pi h}{w\sqrt{D_\K}}\frac{\Gamma(a+1)\zeta(a+1)}{y^{a+1}} = \frac{4\pi^{2-2a}D_\K^{a-\frac{1}{2}}}{y\sin(\pi a)} 
\nonumber\\
&\times\sum_{n=1}^\infty \sigma_{\mathbb{K},-a}(n)\Bigg [ \frac{2^{-2a}\left(-\frac{64\pi^6n^2}{y^2D_\K^2}\right)^{-m}}{\Gamma^2(1-a-2m)}
\bigg \{ {}_1F_4\left(\begin{matrix}
1\\ 1 -\frac{a}{2}-m,  1 -\frac{a}{2}-m,  \frac{1-a}{2}-m, \frac{1-a}{2}-m 
\end{matrix} \bigg | -\frac{4\pi^6n^2}{y^2D_\K^2} \right)\nonumber\\ 
&- 2^{4m}(a+2m)^2(a+2m+1)^2\left(\frac{64\pi^6n^2}{y^2D_\K^2}\right)^{-1} \bigg\}
-\left( \frac{4\pi^6n^2}{y^2D_\K^2} \right)^{\frac{a}{2}}\bigg \{ \cos\left(\frac{\pi a}{2}\right) {\rm ber}\left(4\pi\sqrt{\frac{2n\pi}{yD_{\mathbb{K}}}} \right)
\nonumber\\
&-\sin\left(\frac{\pi a}{2}\right){\rm bei}\left(4\pi\sqrt{\frac{2n\pi}{yD_{\mathbb{K}}}} \right)\bigg \}\Bigg ]+\frac{yD_\K^{a+\frac{3}{2}}}{(2\pi)^{2a+4}\sin(\pi a)}\sum_{k=0}^m\frac{(-1)^k\zeta(2k+2)\zeta_{\mathbb{K}}(2k+a+2)}{\Gamma^2(-a-1-2k)}\left(\frac{8\pi^{3}}{yD_{\mathbb{K}}}\right)^{-2k}.
\end{align}
\end{theorem}
The series $\sum_{n=1}^\infty  \sigma_{\K,a}(n)e^{-ny}$ over any imaginary quadratic field $\K$ appearing in the above theorem, can be considered as an analogue of a series $\sum_{n=1}^\infty  \sigma_{a}(n)e^{-ny}$ in $\Q$ which plays a significant role in the theory of modular forms. For instance, for $a= 2m-1$ with $m \in \N$ and $y = -2\pi i z$ with $z$ lying in the upper half plane, the series in $\Q$  essentially represents the Eisenstein series of weight $2m$ over the full modular group, and for $a= -2m-1$ with $m \in \N$ and $y = -2\pi i z$ the same series in $\Q$ represents the Eichler integral corresponding to the weight $2m+2$ Eisenstein series \cite[Section 5]{BS}. Moreover, the series $\sum_{n=1}^\infty  \sigma_{-1}(n)e^{2\pi i nz}$ appears in the transformation formula of the logarithm of Dedekind eta function \cite[Equation (3.10)]{BLS}. 

In the following theorem, we investigate the transformation for the above series over an imaginary quadratic field $\K$ for $a$ being any natural number.
\begin{theorem}\label{a=2m-1,2m}
For any natural number $m$ and any complex number $\alpha, \beta$  with $\Re(\alpha), \Re(\beta) >0$ and  $\alpha \beta = \frac{D_\K^2}{16\pi^2}$, the transformations
\begin{equation}\label{a=2m-1}
\alpha^m \sum_{n=1}^\infty\sigma_{\mathbb{K},2m-1}(n)e^{-A\pi n\alpha} = - (-\beta)^m \bigg \{ \frac{4}{\sqrt{D_\K}} \sum_{n=1}^\infty\frac{\sigma_{\mathbb{K},1-2m}(n)}{n^{1-2m}}\mathrm{kei}(A\sqrt{\pi n\beta}) +  \frac{\pi h}{w\sqrt{D_\K}}\frac{B_{2m}}{2m}\bigg \}
\end{equation}
and
\begin{equation}\label{a=2m}
\alpha^{m+\frac{1}{2}} \sum_{n=1}^\infty\sigma_{\mathbb{K},2m}(n)e^{-A\pi n\alpha} = \beta^{m+\frac{1}{2}} \bigg \{ \frac{4(-1)^m}{\sqrt{D_\K}} \sum_{n=1}^\infty\frac{\sigma_{\mathbb{K},-2m}(n)}{n^{-2m}}\mathrm{ker}(A\sqrt{\pi n\beta}) +  \frac{h (2m)!}{(2\pi)^{2m}w\sqrt{D_\K}}\zeta(2m+1)\bigg \}
\end{equation}
hold true, where $A =\frac{8\pi}{D_\K}$.
\end{theorem} 
\begin{remark}
One can conclude by a quick observation in the above theorem that \eqref{a=2m-1} provides transformation formula analogous to that for Eisenstein series over an imaginary quadratic field and \eqref{a=2m} provides  an explicit formula for $\zeta(2m+1)$.
\end{remark}


\section{Preliminaries}\label{Prilim}
Throughout the paper, we require some basic tools of analytic number theory and complex analysis. 
\subsection{Schwartz function}
A function is said to be a Schwartz function if all of its derivatives exist and decay faster than any polynomial. We denote the space of Schwartz functions on $\R$ by $\mathscr{S}(\R)$. For $f \in \mathscr{S}(\R)$, we let the Mellin transform of $f$ be $\mathcal{M}(f)$ i.e,
\begin{equation}\label{Schwartz Mellin}
\mathcal{M}(f)(s) = \int_0^\infty f(x)x^{s-1} dx.
\end{equation}
The following lemma provides the analytic behaviour of the Mellin transform of any Schwartz function.
\begin{lemma}\label{Analyticity of Schwartz function}
The function $F(s)$ is absolutely convergent for $\Re(s) > 0$. It can be analytically continued to the whole complex plane except for simple poles at every non-positive integers. It also satisfies the functional equation:
\begin{equation*}\label{F functional equation}
\mathcal{M}(f')(s+1) = -s \, \mathcal{M}(f)(s), 
\end{equation*}
\end{lemma}
\begin{proof}
The functional equation follows immediately from \eqref{Schwartz Mellin}  by applying integration by parts on the integral. Moreover, the functional equation yields
\begin{equation}\label{Eqn:F functional equation}
\mathcal{M}(f^{m})(s+m) = (-1)^{m+1} s(s+1) \cdots (s+m-1) \, \mathcal{M}(f)(s),
\end{equation}
which implies that $\mathcal{M}(f)(s)$ has an analytic continuation to the whole complex plane except for the possible simple poles at $s= 0, 1, \cdots$.
\end{proof}
\textbf{Example.}
One of the most popular example of Schwartz function is $e^{-x}$.  The Mellin transform of $e^{-x}$ is known as Gamma function which can be defined for $\Re(s)>0$ via the convergent improper integral as 
\begin{equation}\label{gammadefn}
\Gamma(s) = \int_0^\infty e^{-x} x^{s-1} {\rm d}x.
\end{equation}
The analytic properties and functional equation of the $\Gamma$-function are given in the following proposition which follows immediately from the previous Lemma.
\begin{proposition}{\cite[Appendix A]{Ayoub}}
The integral in \eqref{gammadefn} is absolutely convergent for $\Re(s) > 0$. It can be analytically continued to the whole complex plane except for simple poles at every non-positive integers. It also satisfies the functional equation:
\begin{equation*}\label{Gamma functional equation}
\Gamma(s+1) = s\Gamma(s).
\end{equation*}
\end{proposition}
The $\Gamma$-function satisfies many important properties. Here we mention two of them.
\begin{itemize}
\item[(i)]
 Euler's reflection formula : 
 \begin{equation}\label{Reflection formula}
 \Gamma(s)\Gamma(1-s) = \frac{\pi}{\sin \pi s}
 \end{equation}
 where $s \notin \mathbb{Z}$.
 \item[(ii)]
  Legendre's duplication formula :
  \begin{equation}\label{Duplication formula}
  \Gamma(s)\Gamma \left(s+\frac{1}{2}\right) = 2^{1-2s} \sqrt{\pi} \Gamma(2s).
  \end{equation}
\end{itemize}
Proofs of these properties can be found in \cite[Appendix A]{Ayoub}.

\subsection{Dedekind zeta function}
 The \begin{it}Dedekind zeta function\end{it} attached to an imaginary quadratic field $\K$ can be defined as 
$$
\zeta_\K(s)=\sum_{\mathfrak{a}\subset\mathcal{O}_\K}\frac{1}{N(\mathfrak{a})^s}=\prod_{\mathfrak{p}\subset \mathcal{O}_\K}\bigg(1-\frac{1}{N(\mathfrak{p})^s}\bigg)^{-1},
$$
for all $s \in \C$ with $\mathfrak{R} (s)>1$, where $\mathfrak{a}$ and $\mathfrak{p}$ run over the non-zero integral ideals and prime ideals of $\mathcal{O}_\K$ respectively.
For $v_\K(m)$ denoting the number of non-zero integral ideals in $\mathcal{O}_\K$ with norm $m$,  $\zeta_\K$ can also be expressed as 
$$
\zeta_\K(s)=\sum_{m=1}^\infty \frac{v_\K(m)}{m^s}.
$$
The following proposition provides the analytic behaviour and the functional equation satisfied by the Dedekind zeta function.
\begin{proposition}{\cite[pp. 254-255]{Lang}}\label{Analyticity of Dedekind zeta}
The function $\zeta_\K(s)$ is absolutely convergent for $\mathfrak{R}(s) > 1$. It can be analytically continued to the whole complex plane except for a simple pole at $s = 1$ with residue $L(1, \chi_D)$. It also satisfies the functional equation
\begin{equation}\label{zetaK functional equation}
\zeta_\K(s) = (2\pi)^{2s-1} D_\K^{\frac{1}{2}-s}\frac{\Gamma(1-s)}{\Gamma(s)}\zeta_\K(1-s).
\end{equation}
\end{proposition}
The famous Dirichlet class number formula for the Dedekind zeta function over an imaginary quadratic field is given in the following proposition.
\begin{proposition}\label{Class number formula}
 The quadratic $L$-function $L(s, \chi_D)$ of $\K$ satisfies
\begin{equation*}
L(1, \chi_D) = \frac{2\pi h}{w\sqrt{D_\K}}.
\end{equation*}
\end{proposition}

\subsection{Special functions}\label{sec:specialfunctions}
The mathematical functions which are non-elementary and are useful due to their applications in mathematical analysis, functional analysis, geometry, physics, and other fields are known as special functions. These mainly appear as solutions of differential equations or integrals of elementary functions. 

One of the most important families of special functions are the Bessel functions.
The Bessel functions of the first kind and the second kind of order $\nu$ are defined by \cite[p.~40, 64]{watson-1944a}
\begin{align*}
	J_{\nu}(z)&:=\sum_{m=0}^{\infty}\frac{(-1)^m(z/2)^{2m+\nu}}{m!\Gamma(m+1+\nu)} \hspace{9mm} (z,\nu\in\mathbb{C}),\\
	Y_{\nu}(z)&:=\frac{J_{\nu}(z)\cos(\pi \nu)-J_{-\nu}(z)}{\sin{\pi \nu}}\hspace{5mm}(z\in\mathbb{C}, \nu\notin\mathbb{Z}),
	\end{align*}
	along with $Y_n(z)=\lim_{\nu\to n}Y_\nu(z)$ for $n\in\mathbb{Z}$. 
The modified Bessel functions of the first and second kinds are defined by \cite[p.~77, 78]{watson-1944a}
\begin{align}
I_{\nu}(z)&:=
\begin{cases}
e^{-\frac{1}{2}\pi\nu i}J_{\nu}(e^{\frac{1}{2}\pi i}z), & \text{if $-\pi<\arg(z)\leq\frac{\pi}{2}$,}\nonumber\\
e^{\frac{3}{2}\pi\nu i}J_{\nu}(e^{-\frac{3}{2}\pi i}z), & \text{if $\frac{\pi}{2}<\arg(z)\leq \pi$,}
\end{cases}\nonumber\\
K_{\nu}(z)&:=\frac{\pi}{2}\frac{I_{-\nu}(z)-I_{\nu}(z)}{\sin\nu\pi}\label{kbesse}
\end{align}
respectively. When $\nu\in\mathbb{Z}$, $K_{\nu}(z)$ is interpreted as a limit of the right-hand side of \eqref{kbesse}. 
The real and imaginary parts of Bessel functions are known as Kelvin functions \cite[p. 267]{NIST}. More precisely, for any $x\geq 0$ and $\nu \in \R$, the Kelvin functions are defined as
\begin{align*}\label{Kelvin Ber Bei}
{\rm ber}_\nu(x) + i \ {\rm bei}_\nu(x) = J_{\nu}(xe^{3\pi i/4})
\end{align*}
and
\begin{align*}
{\rm ker}_{\nu}(x) + i \ {\rm kei}_{\nu}(x) = e^{-\nu \pi i/2}K_{\nu} (xe^{\pi i/4})
\end{align*}
where $J_\nu$ ( resp. $K_\nu$) denotes the Bessel function of first kind (resp. modified Bessel function of second kind) of order $\nu$.

The generalized hypergeometric function is defined by the following power series :
\begin{equation*}
\pFq{p}{q}{a_1, a_2, \cdots, a_p}{b_1,b_2, \cdots, b_q}{z}:=\sum_{n=0}^{\infty}\frac{(a_1)_n(a_2)_n\cdots(a_p)_n}{(b_1)_n(b_2)_n\cdots(b_q)_n}\frac{z^n}{n!}
\end{equation*}
where $(a)_n$ denotes the Pochhammer symbol defined by $(a)_n:=a(a+1)\cdots(a+n-1)=\G(a+n)/\G(a)$.
It is well-known \cite[p.~62, Theorem 2.1.1]{AAR} that the above series converges absolutely for all $z$ if $p\leq q$ and for $|z|<1$ if $p=q+1$, and it diverges for all $z\neq0$ if $p>q+1$ and the series does not terminate.

The following proposition states an important result due to Slater \cite[p.~56-59]{Marichev} which  precisely evaluates inverse Mellin transforms of certain functions in terms of generalized hypergeometric functions. We give its statement below to make the paper self-contained. To begin with we need some notations . Let
\begin{equation*}
\Gamma\genfrac[]{0pt}{}{a_1, a_2, \dots, a_A}{b_1, b_2, \dots, b_B}  \equiv \G[(a);(b)]=\frac{\Gamma \left( a_1 \right) \Gamma \left( a_2 \right) \dots \Gamma \left( a_A \right) }{\Gamma \left( b_1 \right) \Gamma \left( b_2 \right) \dots \Gamma \left( b_B \right) }, 
\end{equation*}
\begin{equation*}
(a) + s := a_1+s, a_2+s, \dots, a_A+s,
\end{equation*}
\begin{equation*}
(b)' - b_k := b_1-b_k, \dots, b_{k-1}-b_k,  b_{k+1}-b_k, \dots, b_B-b_k,
\end{equation*}
\begin{equation*}
\Sigma_A(z) := \sum_{j=1}^A z^{a_j} \Gamma \genfrac[]{0pt}{}{(a)' - a_j, (b) +  a_j}{(c) - a_j, (d) +  a_j} 
{}_{B+C}F_{A+D-1}
 \left( \genfrac{}{}{0pt}{}{(b)+a_j, 1+a_j-(c)}{1+a_j-(a)', (d)+a_j} 
 \bigg| {(-1)^{C-A} z} \right),
\end{equation*}
\begin{equation*}
\Sigma_B(1/z) := \sum_{k=1}^B z^{-b_k} \Gamma \genfrac[]{0pt}{}{(b)' - b_k, (a) +  b_k}{(d) - b_k, (c) +  b_k} 
{}_{A+D}F_{B+C-1}
 \left( \genfrac{}{}{0pt}{}{(a)+b_k, 1+b_k-(d)}{1+b_k-(b)', (c)+b_k} 
 \bigg| {\frac{(-1)^{D-B}}{z}} \right),
\end{equation*}
\begin{proposition}[Slater's Theorem]\label{Slater}
Let 
\begin{equation}\label{functionHs}
\mathscr{H}(s) = \Gamma \genfrac[]{0pt}{}{(a)+s, (b)-s}{(c)+s, (d)-s},
\end{equation}
where the vectors $(a)$, $(b)$, $(c)$, and $(d)$ have, respectively, A, B, C, and D components $a_j$, $b_k$, $c_l$, and $d_m$. Then if the following two groups of conditions hold:
\begin{equation}\label{condition1}
-\textup{Re}(a_j) < \textup{Re}(s)  < \textup{Re}(b_k)  \quad (j = 1,2, \dots, A, \quad  k=1,2, \dots, B),\\
\end{equation}
\begin{equation}\label{condition2}
\begin{cases}
A+B > C+D, \\
A+B = C+D, \quad \textup{Re}(s(A+D-B-C)) < -\textup{Re}(\eta) \\
A=C, \quad B=D, \quad \textup{Re}(\eta) <0,
\end{cases}
\end{equation}
where
\begin{equation*}
\eta := \sum_{j=1}^A a_j + \sum_{k=1}^B b_k - \sum_{l=1}^C c_l - \sum_{m=1}^D d_m,
\end{equation*}
then for these $s$ we have
\begin{equation*}
\mathscr{H}(s) = 
\begin{cases}
\displaystyle\int_{0}^{\infty} x^{s-1} \Sigma_A(x)\, dx, \text{ if } A+D>B+C, \\
\displaystyle\int_{0}^{1} x^{s-1} \Sigma_A(x)\, dx  + \int_{1}^{\infty} x^{s-1} \Sigma_B(1/x)\, dx, \text{ if } A+D=B+C,\\
\displaystyle\int_{0}^{\infty} x^{s-1} \Sigma_B(1/x)\, dx, \text{ if } A+D<B+C,
\end{cases}
\end{equation*}
$\Sigma_A(1) = \Sigma_B(1)$ if $A+D=B+C$, \textup{Re}$(\eta) +C-A+1<0, A \geq C$.
\end{proposition}
\begin{corollary}\cite[p.~58]{Marichev} \label{SlatersCor}
Under the conditions \eqref{condition1} and \eqref{condition2}, the inverse Mellin transform of the function in \eqref{functionHs} is a function $H(x)$ of hypergeometric type given by
\begin{equation*}
H(x) = 
\begin{cases}
\Sigma_A(x) \text{ for } x>0, \quad \text{ if } A+D>B+C, \\
\Sigma_A(x) \text{ for } 0<x<1, \quad \text{ or }  \quad \Sigma_B(1/x) \text{ for } x>1, \quad  \text{ if } A+D=B+C,\\
\Sigma_B(1/x) \text{ for } x>0, \quad \text{ if } A+D<B+C,
\end{cases}
\end{equation*}
$\mathscr{H}(1) = \Sigma_A(1) = \Sigma_B(1)$ if $A+D=B+C$, Re$(\eta) +C-A+1<0, A \geq C$.
\end{corollary}

\section{Generalization of a Voronoi-type identity over an imaginary quadratic field}
In this section, we setup our main ingredients to prove the identities provided in \S \ref{intro}. 
Dirichlet introduced the problem of counting the number of lattice points inside or on the hyperbola. In other words, he studied the asymptotic behaviour of the summatory function of the divisor function. Let  $d(n)$ denotes the divisor function i.e, $d(n) = \sum_{d\mid n} 1$. He obtained an asymptotic formula with the main term $x\log x + (2\gamma - 1)x + \frac{1}{4}$ and an error term of order $\sqrt{x}$. The problem of estimating the error term is known as the Dirichlet hyperbola problem or the Dirichlet divisor problem. The bound on the error term has been further improved by many mathematicians. At this writing, the best estimate $O(x^{131/416+\epsilon})$, for each $\epsilon > 0$, as $x \to \infty$, is due to M.~N.~ Huxley \cite{Huxley}.

Vorono\"{\dotlessi} \cite{Voronoi} introduced a new phase into the Dirichlet divisor problem. He was able to express the error term as an infinite series containing the Bessel functions. More precisely, letting $Y_\nu$ (resp. $K_\nu$) denote the Bessel function of the second kind (resp. modified Bessel function of second kind) of order $\nu$ and $\gamma$ denote the Euler constant, a celebrated identity of Vorono\"{\dotlessi} is given by
\begin{equation}\label{Voronoi 1-identity}
\sideset{}{'}\sum_{n\le x}\!\!d(n) =x\log x+(2\gamma-1)x+\frac{1}{4} - \sum_{k=1}^\infty\frac{d(k)}{k}
\left(Y_1\left(4\pi\,\sqrt{xk}\,\right)+\frac{2}{\pi}K_1\left(4\pi\,\sqrt{xk}\,\right)\right)\sqrt{xk},
\end{equation}
where $\sum'$ means that the term corresponding to $n=x$ is halved. In the same article \cite{Voronoi}, Vorono\"{\dotlessi} also obtained a more general form of \eqref{Voronoi 1-identity}, namely
\begin{equation}\label{Voronoi 2-identity}
\sum_{\alpha<n<\beta} d(n) f(n) = \int_\alpha^\beta (2\gamma + \log t) f(t) {\rm d}t + 2\pi \sum_{n=1}^\infty d(n) \int_\alpha^\beta f(t) \left(\frac{2}{\pi} K_0(4\pi \sqrt{nt}) - Y_0(4\pi \sqrt{nt}) \right) {\rm d}t,
\end{equation}
where $f(t)$ is a function of bounded variation in $(\alpha, \beta)$ and $0<\alpha<\beta$.  A shorter proof of the above identity for $0<\alpha<\beta$ with $\alpha, \beta \not\in \Z$ was offered by Koshliakov in \cite{Koshliakov} where he assumed $f$ to be any analytic function lying inside a closed contour strictly containing the interval $[\alpha, \beta]$. The identity \eqref{Voronoi 1-identity} can be generalized by generalizing the divisor function in different directions (cf. \cite{Banerjee} \cite{WB}).


The identity \eqref{Voronoi 2-identity} was generalized in \cite[Section 6, 7]{bdrz} for the general divisor function $\sigma_a(n)$ which can be defined as $\sigma_a(n) : = \sum_{d\mid n} d^a$ where $a$ is any complex number. 
The function $\sigma_{\K, a}(n)$ defined in \eqref{Divisor function} is basically the function which is analogous to $\sigma_a(n)$ over an imaginary quadratic field.  The following theorem states an analogous identity of \eqref{Voronoi 2-identity} associated to the divisor function $\sigma_{\K, a}(n)$. To the best of our knowledge, the result is new. Before stating our result, we define the function 
\begin{align*}
H_{\K, \nu}(x): = \frac{\sqrt{\pi}}{\sin(2\pi\nu)}&\Bigg\{\frac{2^{1-4\nu}}{\Gamma^2(1-2\nu)}\left(\frac{x}{4}\right)^{-\nu}\pFq05{-}{1-\nu,1-\nu,\frac{1}{2}-\nu,\frac{1}{2}-\nu,\frac{1}{2}}{-\frac{x^2}{16}}\nonumber\\
&-\frac{2^{1+2\nu}\cos(\pi\nu)}{\Gamma(1+2\nu)}\left(\frac{x}{4}\right)^{\nu}\pFq05{-}{1+\nu,\frac{1}{2}+\nu,\frac{1}{2},\frac{1}{2},1}{-\frac{x^2}{16}}\nonumber\\
&-\frac{2^{4+2\nu}\sin(\pi\nu)}{\Gamma(2+2\nu)}\left(\frac{x}{4}\right)^{1+\nu}\pFq05{-}{\frac{3}{2}+\nu,1+\nu,\frac{3}{2},\frac{3}{2},1}{-\frac{x^2}{16}}\Bigg\} .
\end{align*}

\begin{theorem}\label{Voronoi over K}
Let $a$ be any complex number with $-1<\mathrm{Re}(a)<1$. Then for any Schwarz function $f$, the identity
\begin{align*}
\sum_{n=1}^\infty\sigma_{\K,a}(n)f(n)&=\int_0^\infty \left(\zeta_\K(1-a)+ t^a \frac{2\pi h \zeta(1+a)}{w\sqrt{D_\K}} \right)  f(t) \, {\rm d}t  -\frac{1}{2} \zeta_\K(-a)f(0^+) \nonumber\\
&+2 \pi^{ \frac{3-a}{2}}D_{\K}^{\frac{a-1}{2}}\sum_{n=1}^\infty \sigma_{\K,-a}(n)n^{a/2} \int_0^\infty t^{a/2}H_{\K, a/2}\left(\frac{4\pi^{3}nt}{D_\K}\right)f(t)\ {\rm d}t.
\end{align*}
holds, provided the Mellin transform of $f$ decays faster than any polynomial in any bounded vertical strip.
\end{theorem}
%
\begin{proof}
For Re$(s)>1$ and Re$(s-a)>1$, the Dirichlet series associated to the divisor function function $\sigma_{\K, a}(n)$ is given by
\begin{align}\label{Dirichlet series of zeta zetak}
\sum_{n=1}^\infty \frac{\sigma_{\K, a}(n)}{n^s} = \zeta(s)\zeta_{\K}(s-a).
\end{align}
For $f \in \mathscr{S}(\R)$, its inverse Mellin transform on $F$ yields 
\begin{equation}\label{Eqn:Voronoi sum}
I_{\K, a} = \sum_{n=1}^{\infty} \sigma_{\K, a}(n) f(n) = \sum_{n=1}^{\infty} \sigma_{\K, a}(n) \frac{1}{2\pi i} \int_{(c)} F(s) n^{-s} {\rm d}s = \frac{1}{2\pi i}\int_{(c)}F(s)\zeta(s)\zeta_{\K}(s-a) {\rm d}s 
\end{equation}
where $ c > \max (1, 1+\Re(a))$. 
We next consider the contour $\mathcal{C}$ given by the rectangle with vertices $\{c - iT,c + iT, \lambda + iT, \lambda - iT\}$ in the anticlockwise direction as $T \to \infty$ where $-1<\lambda<0$. It follows from Lemma \ref{Analyticity of Schwartz function}, the analytic behaviour of $\zeta(s)$ and Proposition \ref{Analyticity of Dedekind zeta} that the integrand is analytic inside the contour except for the possible simple poles at $s = 0, 1$ and $1+a$. Employing the Cauchy residue theorem, we have
\begin{equation}\label{Eqn:Cauchy theorem}
\frac{1}{2\pi i}\int_{\mathcal{C}}F(s)\zeta(s)\zeta_{\K}(s-a) {\rm d}s = \mathcal{R}_0 + \mathcal{R}_1 + \mathcal{R}_{1+a}
\end{equation}
where  $\mathcal{R}_{z_0}$ denotes the residue of the integrand at $z_0$. We next evaluate the values of $\mathcal{R}_0$, $\mathcal{R}_1 $ and $\mathcal{R}_{1+a}$ using Lemma \ref{Analyticity of Schwartz function}, Proposition \ref{Analyticity of Dedekind zeta} and \ref{Class number formula} respectively, which are given by
$$\mathcal{R}_0 = \lim_{s \to 0} s F(s)\zeta(s)\zeta_{\K}(s-a) = \frac{1}{2} \mathcal{M}(f')(1) \zeta_{\K}(-a) = \frac{\zeta_{\K}(-a)}{2} \int_0^\infty f'(t) \, {\rm d}t =- \frac{\zeta_{K}(-a) f(0^+)}{2},$$
$$\mathcal{R}_1 = \lim_{s \to 1} (s-1) F(s)\zeta(s)\zeta_{\K}(s-a) = F(1)\zeta_\K(1-a) = \zeta_\K(1-a) \int_0^\infty f(t) \, {\rm d}t$$
and
$$\mathcal{R}_{1+a} = \lim_{s \to 1+a} (s-1-a) F(s)\zeta(s)\zeta_{\K}(s-a) = F(1+a)\zeta(1+a)\frac{2\pi h}{w\sqrt{D_\K}} = \frac{2\pi h \zeta(1+a)}{w\sqrt{D_\K}} \int_0^\infty f(t) t^a \, {\rm d}t$$
Inserting the values of $\mathcal{R}_0$, $\mathcal{R}_1 $ and $\mathcal{R}_{1+a}$ in \eqref{Eqn:Cauchy theorem}, the equations \eqref{Eqn:Voronoi sum} and \eqref{Eqn:Cauchy theorem} together imply
\begin{equation}\label{Hori and Vert Int}
I_{\K, a} = \int_0^\infty \left(\zeta_\K(1-a)+ t^a \frac{2\pi h \zeta(1+a)}{w\sqrt{D_\K}} \right)  f(t) \, {\rm d}t -\frac{1}{2} \zeta_\K(-a)f(0^+) + \mathcal{H}_1 + \mathcal{H}_2 + \mathcal{V}
\end{equation}
where $\mathcal{H}_1 := \lim\limits_{T \to \infty} \frac{1}{2\pi i}\int_{\mu + iT}^{c+iT} F(s)\zeta(s)\zeta_{\K}(s-a) \, {\rm d}s$ and $\mathcal{H}_2 := \lim\limits_{T \to \infty} \frac{1}{2\pi i}\int_{c - iT}^{\mu - iT} F(s)\zeta(s)\zeta_{\K}(s-a) \, {\rm d}s$ are the horizontal integrals and $\mathcal{V} := \frac{1}{2\pi i}\int_{(\lambda)} F(s)\zeta(s)\zeta_{\K}(s-a) \, {\rm d}s$ is the vertical integral.

It follows from a standard argument of the Phragmen-Lindel{\"o}f principle [cf. \cite[Chapter 5]{Iwaniec}] and the functional equation of both zeta functions that for $s = \sigma + it$ with $\lambda < \sigma < c$ and for some $\theta \in \R$,
$$ |\zeta(\sigma+it)\zeta_{\K}(\sigma+it)| \ll t^{\theta (1-\sigma)},\quad \mathrm{as}\ t\to\infty.$$
On the other hand according to our hypothesis,  $F(s)$ decays faster than any polynomial in $t$ in the above vertical strip. Thus, the horizontal integrals $\mathcal{H}_1$ and $\mathcal{H}_2$ vanish.

We next concentrate on the vertical integral $\mathcal{V}$. The functional equation of the Riemann zeta function
\begin{align*}
\zeta(s)&=2^s\pi^{s-1}\Gamma(1-s)\sin\left(\frac{\pi s}{2}\right)\zeta(1-s)
\end{align*}
and that of the Dedekind zeta function in \eqref{zetaK functional equation}
together imply that
\begin{align}\label{functprod}
\zeta(s)\zeta_{\K}(s-a)=D_\K^{\frac{1}{2}-s+a}2^{3s-2a-1}\pi^{3s-2a-2}\frac{\Gamma(1-s)\Gamma(1-s+a)}{\Gamma(s-a)}\sin\left(\frac{\pi s}{2}\right) \zeta(1-s)\zeta_\K(1-s+a).
\end{align}
Substituting \eqref{functprod} into $\mathcal{V}$ and changing the variable $s$ by $1-s$ in the next step, the vertical integral becomes
\begin{align*}
\mathcal{V}&=\frac{2D_\K^{a+\frac{1}{2}}}{(2\pi)^{2a+2}}\frac{1}{2\pi i}\int_{(\lambda)}F(s)\frac{\Gamma(1-s)\Gamma(1-s+a)}{\Gamma(s-a)}\sin\left(\frac{\pi s}{2}\right)\zeta(1-s)\zeta_\K(1-s+a)\left(\frac{8\pi^3}{D_\K}\right)^s\ ds\nonumber\\
&=\frac{2D_\K^{a-\frac{1}{2}}}{(2\pi)^{2a-1}}\frac{1}{2\pi i}\int_{(1-\lambda)}\frac{F(1-s)\Gamma(s)\Gamma(s+a)}{\Gamma(1-s-a)}\cos\left(\frac{\pi s}{2}\right)\zeta(s)\zeta_\K(s+a)\left(\frac{8\pi^3}{D_\K}\right)^{-s} ds.
\end{align*}
We now replace $s$ by $s-a$ and assume $\lambda^*:=1-\lambda+\Re(a)$ in the above integral to obtain
\begin{align}\label{z1}
\mathcal{V}=\frac{2(2\pi)^{a+1}}{\sqrt{D_\K}}\frac{1}{2\pi i}\int_{(\lambda^*)}\frac{F(1+a-s)\Gamma(s-a)\Gamma(s)}{\Gamma(1-s)}\cos\left(\frac{\pi }{2}(s-a)\right)\zeta(s-a)\zeta_\K(s)\left(\frac{8\pi^3}{D_\K}\right)^{-s}\ ds.
\end{align}
For Re$(s)>1$ and Re$(s-a)>1$, it follows that
\begin{align*}
\zeta(s-a)\zeta_\K(s)=\sum_{n=1}^\infty\frac{\sigma_{\K,-a}(n)}{n^{s-a}},
\end{align*}
therefore the integral \eqref{z1} can be written as
\begin{align}\label{Ix1}
\mathcal{V} = \frac{2(2 \pi)^{a+1}}{\sqrt{D_\K}}\sum_{n=1}^\infty \frac{\sigma_{\K,-a}(n)}{n^{-a}}I_{\K,a}(n)
\end{align}
where,
\begin{align}\label{ikan}
I_{\K,a}(n):=\frac{1}{2\pi i}\int_{(\lambda^*)}F(1+a-s)N_{\K,a}(s)\left(\frac{8\pi^3n}{D_\K}\right)^{-s}\ ds,
\end{align}
 and
\begin{align*}
N_{\K,a}(s):=\frac{\Gamma(s-a)\Gamma(s)}{\Gamma(1-s)}\cos\left(\frac{\pi }{2}(s-a)\right).
\end{align*}
We apply \eqref{Reflection formula} and \eqref{Duplication formula} together on the above factor $N_{\K,a}(s)$ to obtain
\begin{align}\label{nka2s}
N_{\K,a}(s)&=2^{3s-a-2}\pi^{\frac{1}{2}}\frac{\Gamma\left(\frac{s}{2}-\frac{a}{2}\right)\Gamma(\frac{s}{2})\Gamma\left(\frac{s}{2}+\frac{1}{2}\right)}{\Gamma\left(\frac{1}{2}-\frac{s}{2}\right)\Gamma(1-\frac{s}{2})\Gamma\left(\frac{1}{2}+\frac{a}{2}-\frac{s}{2}\right)}.
\end{align}
On the other hand, using \eqref{Eqn:F functional equation} into the integral \eqref{ikan}, we evaluate
\begin{align}\label{beforeintbyparts}
I_{\K,a}(n)&=-\frac{1}{2\pi i}\int_{(\lambda^*)}\int_0^\infty\frac{N_{\K,a}(s)f'(t)t^{1+a-s}}{1+a-s}\left(\frac{8\pi^3 n}{D_\K}\right)^{-s}\ dt \, ds\nonumber\\
&=-\frac{1}{n^{1+a}}\int_0^\infty f'(t)\left(\frac{1}{2\pi i}\int_{(\lambda^*)}\frac{N_{\K,a}(s)(nt)^{1+a-s}}{1+a-s}\left(\frac{8\pi^3}{D_\K}\right)^{-s}\ ds\right)\ dt\nonumber\\
&=-\frac{1}{n^{1+a}}\int_0^\infty f'(t)J_{\K,a}(nt)\ dt,
\end{align}
where
\begin{align*}
J_{\K,a}(x):=\frac{1}{2\pi i}\int_{(\lambda^*)}\frac{N_{\K,a}(s)x^{1+a-s}}{1+a-s}\left(\frac{8\pi^3}{D_\K}\right)^{-s}\ ds.
\end{align*}
We perform integration by part in \eqref{beforeintbyparts} considering $J_{\K,a}(nt)$ as first function and $f'(t)$ as second to obtain
\begin{align}\label{afterintbyparts}
I_{\K,a}(n)&=\frac{1}{n^{a+1}}\int_0^\infty f(t)\frac{d}{dt}\left( J_{\K,a}(nt)\right)\ dt.
\end{align}
Differentiating $J_{\K,a}(nt)$ with respect to $t$, we get
\begin{align}\label{beforenka2s}
\frac{d}{dt}\left(J_{\K,a}(nt)\right)&=\frac{n^{a+1}t^a}{2\pi i}\int_{(\lambda^*)}N_{\K,a}(s)\left(\frac{8\pi^3nt}{D_\K}\right)^{-s}\ ds.
\end{align}
We next insert the factor $N_{\K, a}(s)$ from \eqref{nka2s} and replace $s$ by $\frac{a}{2}-2s$ into \eqref{beforenka2s} to deduce that
\begin{align*}
\frac{d}{dt}\left(J_{\K,a}(nt)\right)&=\frac{n^{a/2+1}(tD_\K)^{a/2}}{2^{a+1}\pi^{\frac{3a-1}{2}}}\frac{1}{2\pi i}\int_{\left(-\frac{\lambda^*}{2} + \frac{a}{4}\right)}\frac{\Gamma\left(-\frac{a}{4}-s\right)\Gamma(\frac{a}{4}-s)\Gamma\left(\frac{1}{2}+\frac{a}{4}-s\right)}{\Gamma\left(\frac{1}{2}-\frac{a}{4}+s\right)\Gamma(1-\frac{a}{4}+s)\Gamma\left(\frac{1}{2}+\frac{a}{4}+s\right)}\left(\frac{\pi^{6}n^2t^2}{D_\K^2}\right)^s ds.
\end{align*}
Invoking Proposition \ref{Slater} and applying \eqref{Reflection formula} and \eqref{Duplication formula} both in the next step, we write the above integral as
\begin{align}\label{Apply Slater's theorem}
\frac{1}{2\pi i}&\int_{\left(-\frac{\lambda^*}{2} + \frac{a}{4}\right)}\frac{\Gamma\left(-\frac{a}{4}-s\right)\Gamma(\frac{a}{4}-s)\Gamma\left(\frac{1}{2}+\frac{a}{4}-s\right)}{\Gamma\left(\frac{1}{2}-\frac{a}{4}+s\right)\Gamma(1-\frac{a}{4}+s)\Gamma\left(\frac{1}{2}+\frac{a}{4}+s\right)}\left(\frac{\pi^{6}n^2t^2}{D_\K^2}\right)^s ds\nonumber\\
&= \frac{\Gamma(\frac{a}{2})\Gamma\left(\frac{1}{2}+\frac{a}{2}\right)}{\Gamma\left(\frac{1}{2}\right)\Gamma(1-\frac{a}{2})\Gamma\left(\frac{1}{2}-\frac{a}{2}\right)}\left(\frac{\pi^3nt}{D_\K}\right)^{-\frac{a}{2}}\pFq05{-}{1-\frac{a}{2},1-\frac{a}{2},\frac{1}{2}-\frac{a}{2},\frac{1}{2}-\frac{a}{2},\frac{1}{2}}{-\frac{\pi^6n^2t^2}{D_\K^2}}\nonumber\\
&+\frac{\Gamma(-\frac{a}{2})\Gamma\left(\frac{1}{2}\right)}{\Gamma\left(\frac{1}{2}+\frac{a}{2}\right)\Gamma\left(\frac{1}{2}\right)}\left(\frac{\pi^3nt}{D_\K}\right)^{\frac{a}{2}}\pFq05{-}{1+\frac{a}{2},\frac{1}{2}+\frac{a}{2},\frac{1}{2},\frac{1}{2},1}{-\frac{\pi^6n^2t^2}{D_\K^2}}\nonumber\\
&+\frac{\Gamma(-\frac{1}{2}-\frac{a}{2})\Gamma\left(-\frac{1}{2}\right)}{\Gamma\left(1+\frac{a}{2}\right)\Gamma\left(\frac{3}{2}\right)}\left(\frac{\pi^3nt}{D_\K}\right)^{1+\frac{a}{2}}\pFq05{-}{\frac{3}{2}+\frac{a}{2},1+\frac{a}{2},\frac{3}{2},\frac{3}{2},1}{-\frac{\pi^6n^2t^2}{D_\K^2}}\nonumber\\
&= H_{\K, a/2}\left(\frac{4\pi^3nt}{D_\K}\right). 
\end{align}
Thus the derivative of $J_{\K, a}(nt)$ reduces to
\begin{equation}\label{Jkant final}
\frac{d}{dt}\left(J_{\K,a}(nt)\right) = \frac{n^{a/2+1}(tD_\K)^{a/2}}{2^{a+1}\pi^{\frac{3a-1}{2}}} H_{\K, a/2}\left(\frac{4\pi^3nt}{D_\K}\right). 
\end{equation}
Employing \eqref{Jkant final} into \eqref{afterintbyparts} and inserting the resulting expression into \eqref{Ix1}, we evaluate the vertical integral as
\begin{equation}\label{Vert Int eval}
\mathcal{V} = 2 \pi^{ \frac{3-a}{2}}D_{\K}^{\frac{a-1}{2}}\sum_{n=1}^\infty \sigma_{\K,-a}(n)n^{a/2} \int_0^\infty t^{a/2}H_{\K, a/2}\left(\frac{4\pi^{3}nt}{D_\K}\right)f(t)\ {\rm d}t
\end{equation}
Finally, the above evaluation \eqref{Vert Int eval} and equation \eqref{Hori and Vert Int} together concludes our theorem.
\end{proof}

\section{Identities for the Dedekind zeta function over an imaginary quadratic field}
In this section, we mainly investigate the transformation formulas for the series $\sum_{n=1}^\infty \sigma_{\K, a}(n)e^{-ny}$, where $a$ and $y$ are any complex numbers with $\Re(y) >0$. The following lemma provides the growth of the function which is mainly involved inside the series of right hand side of Theorem \ref{Lambert series}. It plays a significant role in proving Theorem \ref{Lambert series} and Theorem \ref{Analytic continuation}.  
\begin{lemma}\label{Asymptotic for summand}
For any complex number $a$ and any non-negative integer $m$, we have
\begin{align}\label{Asymptotics for RHS}
&\frac{2^{-2a}}{\Gamma^2(1-a)}\pFq14{1}{1-\frac{a}{2},1-\frac{a}{2},\frac{1-a}{2},\frac{1-a}{2}}{-z} 
- z^{a/2} \left( \cos\left(\frac{\pi a}{2}\right) {\rm ber}(4z^{1/4})-  \sin\left(\frac{\pi a}{2}\right) {\rm bei}(4z^{1/4}) \right)\nonumber\\
&=\frac{1}{2^{2a}} \sum_{k=0}^m \frac{(-1)^k (16z)^{-k-1}}{\Gamma^2(-1-a-2k)} + O\left(\frac{1}{z^{m+2}}\right).
\end{align}
\end{lemma}

\begin{proof}
We first apply
\begin{align}\label{BerBei}
{\rm ber}(4z^{1/4}) = \pFq03{-}{\frac{1}{2},\frac{1}{2},1}{-z} \hspace{.5cm} \text{and} \hspace{.5cm}  {\rm bei}(4z^{1/4}) = 4\sqrt{z} \, \ \pFq03{-}{\frac{3}{2},\frac{3}{2},1}{-z}
\end{align}
[cf. \cite[Formula (13), (17), p. 516]{Prudnikov}] together to write the left-hand side of \eqref{Asymptotics for RHS} as
\begin{align*}
& \frac{2^{-2a}}{\Gamma^2(1-a)}\pFq14{1}{1-\frac{a}{2},1-\frac{a}{2},\frac{1-a}{2},\frac{1-a}{2}}{-z} 
- z^{a/2} \left( \cos\left(\frac{\pi a}{2}\right) {\rm ber}(4z^{1/4}) - \sin\left(\frac{\pi a}{2}\right) {\rm bei}(4z^{1/4}) \right) 
\nonumber\\
&= \frac{\sin(\pi a) z^{\frac{a}{4}}}{2\pi} \Bigg\{\frac{\Gamma\left(\frac{a}{2}\right)\Gamma\left(\frac{1+a}{2}\right)}{\Gamma\left(1-\frac{a}{2}\right)\Gamma\left(\frac{1-a}{2}\right)}z^{-\frac{a}{4}}\pFq14{1}{1-\frac{a}{2},1-\frac{a}{2},\frac{1-a}{2},\frac{1-a}{2}}{-z}\nonumber\\
&+\Gamma\left(-\frac{a}{2}\right)\Gamma\left(1+\frac{a}{2}\right)z^{\frac{a}{4}}\pFq03{-}{\frac{1}{2},\frac{1}{2},1}{-z}
-4\Gamma\left(-\frac{1}{2}-\frac{a}{2}\right)\Gamma\left(\frac{3+a}{2}\right)z^{\frac{1}{2}+\frac{a}{4}}\pFq03{-}{\frac{3}{2},\frac{3}{2},1}{-z} \Bigg\}
\end{align*}
Invoking Proposition \ref{Slater}, the above equation reduces to
\begin{align}\label{Inverse Slater}
&\frac{2^{-2a}}{\Gamma^2(1-a)}\pFq14{1}{1-\frac{a}{2},1-\frac{a}{2},\frac{1-a}{2},\frac{1-a}{2}}{-z} 
- z^{a/2} \left( \cos\left(\frac{\pi a}{2}\right) {\rm ber}(4z^{1/4}) - \sin\left(\frac{\pi a}{2}\right) {\rm bei}(4z^{1/4}) \right)\nonumber\\
&= \frac{\sin(\pi a) z^{\frac{a}{4}}}{2\pi} \left\{ \frac{1}{2\pi i} \int_{(\eta)} \frac{\Gamma(1+\frac{a}{4}+s) \Gamma(-\frac{a}{4}-s) \Gamma(\frac{a}{4}-s) \Gamma(\frac{1}{2}+\frac{a}{4}-s)}{\Gamma(\frac{1}{2}-\frac{a}{4}+s) \Gamma(1-\frac{a}{4}+s)} z^{s} \, ds \right\}.
\end{align}
where $-1-\frac{\Re(a)}{4} < \eta < \min\left\{\pm \frac{\Re(a)}{4}, \frac{1}{2} + \frac{\Re(a)}{4} \right\}$. The definition of  Meijer G-function \cite[p.~143]{Luke} readily implies that the above integral can be expressed as
\begin{align}\label{Apply G definition}
\frac{1}{2\pi i} \int_{(\eta)} \frac{\Gamma(1+\frac{a}{4}+s) \Gamma(-\frac{a}{4}-s) \Gamma(\frac{a}{4}-s) \Gamma(\frac{1}{2}+\frac{a}{4}-s)}{\Gamma(\frac{1}{2}-\frac{a}{4}+s) \Gamma(1-\frac{a}{4}+s)} z^{s} \, ds &= \MeijerG*{3}{1}{1}{5}{-\frac{a}{4}}{-\frac{a}{4},\frac{a}{4},\frac{1}{2}+\frac{a}{4};\frac{a}{4},\frac{1}{2}+\frac{a}{4}}{z}.
\end{align}
We next find the asymptotics of Meijer G-function. For $1\leq h \leq p <q$, $1 \leq g \leq q$ and $|\arg(z)|\leq \rho \pi-\delta$ with $\rho > 0$ and $\delta \geq 0$, it follows from \cite[Theorem 2, p. 179]{Luke} that for $|z| \to \infty$, we have
\begin{align}\label{asyofg}
\MeijerG*{g}{h}{p}{q}{a_1,\cdots,a_p}{b_1,\cdots,b_q}{z}\sim \sum_{j=1}^h\exp(-i\pi(\nu+1)a_j)\Delta_q^{g,h}(j)E_{p,q}\left(z\exp(i\pi(\nu+1)||a_j\right),
\end{align}
where $\nu=q-g-h$,
\begin{align*}
&E_{p,q}(z||a_j):=\frac{z^{a_j-1}\prod\limits_{\ell=1}^q\Gamma(1+b_\ell-a_j)}{\prod\limits_{\ell=1}^p\Gamma(1+a_\ell-a_j)}\sum_{k=0}^m\frac{\prod\limits_{\ell=1}^q\left(1+b_\ell-a_j\right)_k}{k!\prod\limits_{\substack{\ell=1\\\ell\neq j}}^p(1+a_p-a_j)_k}\left(-\frac{1}{z}\right)^k + O\left(\frac{1}{z^{m+2-a_j}} \right) \nonumber\\
\text{and }\hspace{.2cm} &\Delta_q^{g,h}(j):=(-1)^{\nu+1}\left(\prod\limits_{\substack{\ell=1\\ \ell\neq j}}^h\Gamma(a_\ell-a_j)\Gamma(1+a_\ell-a_j)\right) 
\Bigg/\left(\prod\limits_{\ell=g+1}^q\Gamma(a_j-b_\ell)\Gamma(1+b_\ell-a_j)\right).
\end{align*}
Letting $g=3, h= p = 1$ and $q = 5$ in \eqref{asyofg}, we have
\begin{align}\label{Asymptotics for G3115}
\MeijerG*{3}{1}{1}{5}{-\frac{a}{4}}{-\frac{a}{4},\frac{a}{4},\frac{1}{2}+\frac{a}{4};\frac{a}{4},\frac{1}{2}+\frac{a}{4}}{z}
= \frac{\Gamma(1+\frac{a}{2})\Gamma(\frac{3+a}{2})z^{-\frac{a}{4}-1}}{\Gamma(-\frac{a}{2})\Gamma(-\frac{a}{2}-\frac{1}{2})} \sum_{k=0}^m (-1)^k \left(1+\frac{a}{2}\right)_k^2 \left(\frac{3+a}{2}\right)_k^2 z^{-k} + O\left( \frac{1}{z^{m+\frac{a}{4}+2}} \right)
\end{align}
for $n \to \infty$. Finally \eqref{Inverse Slater}, \eqref{Apply G definition} and \eqref{Asymptotics for G3115} together with the application of \eqref{Reflection formula} and \eqref{Duplication formula} on the gamma factors inside the integral, conclude our Lemma.
\end{proof}

\subsection{Proof of Theorem \ref{Lambert series}} We first prove the result for $0 < \Re(a) < 1$ and $y > 0$, later we extend it to $\Re(a) > −1$ and $Re(y) > 0$ respectively by analytic continuation. 
We consider the particular Schwartz function $f(n) = e^{-ny}$ with $y>0$ in Theorem \ref{Voronoi over K}, which yields the following identity
\begin{align}\label{Apply Voronoi}
\sum_{n=1}^\infty\sigma_{\K,a}(n)e^{-ny}=\frac{\zeta_{\K}(1-a)}{y} &+ \frac{ 2\pi h\Gamma(a+1)\zeta(a+1)}{y^{a+1}w\sqrt{D_\K}} -\frac{1}{2} \zeta_\K(-a)+2 \pi^{ \frac{3-a}{2}}D_{\K}^{\frac{a-1}{2}}\sum_{n=1}^\infty \frac{\sigma_{\K,-a}(n)}{n^{-a/2}} I_{\K, a}(n)
\end{align}
where 
$I_{\K, a}(n) = \int_0^\infty t^{a/2}H_{\K, a/2}\left(\frac{4\pi^{3}nt}{D_\K}\right)e^{-ty}\ {\rm d}t.$
We now concentrate on simplifying the integral $I_{\K, a}(n)$. Considering two functions $h_1(t) = t^{a/2}e^{-ty}$ and $h_2(t) = H_{\K, a/2}\left(\frac{4\pi^{3}nt}{D_\K}\right)$, the integral can be expressed in the form
$$I_{\K, a}(n) := \int_0^\infty  h_1(t)h_2(t) \ {\rm d}t.$$ 
The Mellin transform associated to $h_1(t)$ and $h_2(t)$ is denoted by $H_1(s)$ and $H_2(s)$ respectively, which we need to evaluate next. We first obtain the Mellin transform of $h_1(t)$ as
\begin{align*}
H_1(s) := \int_0^\infty  h_1(t) t^{s-1} dt 
= \int_0^\infty e^{-ty} t^{a/2+s-1} dt 
= \frac{\Gamma(a/2+s)}{y^{a/2+s}}
\end{align*}
where the integral is valid for $\Re(s) > -\frac{\Re(a)}{2}$. It follows from \eqref{Apply Slater's theorem} that 
\begin{align*}
h_2(t) &:= \frac{1}{2\pi i} \int_{\left(-\frac{\lambda^*}{2} + \frac{a}{4}\right)}\frac{\Gamma\left(-\frac{a}{4}-s\right)\Gamma(\frac{a}{4}-s)\Gamma\left(\frac{1}{2}+\frac{a}{4}-s\right)}{\Gamma\left(\frac{1}{2}-\frac{a}{4}+s\right)\Gamma(1-\frac{a}{4}+s)\Gamma\left(\frac{1}{2}+\frac{a}{4}+s\right)}\left(\frac{\pi^{6}n^2t^2}{D_\K^2}\right)^s ds\\
&= \frac{1}{4\pi i} \int_{(\lambda^*-a)} \frac{\Gamma(-\frac{a}{4}+ \frac{s}{2}) \Gamma(\frac{a}{4}+ \frac{s}{2}) \Gamma(\frac{1}{2} + \frac{a}{4}+ \frac{s}{2})}{\Gamma(\frac{1}{2} - \frac{a}{4} - \frac{s}{2}) \Gamma(1 - \frac{a}{4}- \frac{s}{2}) \Gamma(\frac{1}{2} + \frac{a}{4}- \frac{s}{2})} \left(\frac{\pi^3nt}{D_\K} \right)^{-s}  ds.
\end{align*}
Thus the Mellin transform of $h_2(t)$ can be evaluated as
\begin{align*}
H_2(s) := \int_0^\infty  h_2(t) t^{s-1} dt = \frac{\Gamma(-\frac{a}{4}+ \frac{s}{2}) \Gamma(\frac{a}{4}+ \frac{s}{2}) \Gamma(\frac{1}{2} + \frac{a}{4}+ \frac{s}{2})}{2\Gamma(\frac{1}{2} - \frac{a}{4} - \frac{s}{2}) \Gamma(1 - \frac{a}{4}- \frac{s}{2}) \Gamma(\frac{1}{2} + \frac{a}{4}- \frac{s}{2})} \left(\frac{\pi^3n}{D_\K} \right)^{-s}
\end{align*}
where the integral is valid for $1<\Re(s)<2$. On the other hand, the region of convergence for $H_1(1-s)$ is $\Re(s) < 1+\frac{\Re(a)}{2}$. Thus applying Parseval's formula [cf. \cite[p. 83]{Paris}] for any real $\mu$ satisfying $1 < \mu < 1+\frac{\Re(a)}{2}$, we obtain
\begin{align}\label{Eqn:IKan}
I_{\K, a}(n) = \frac{1}{2y^{a/2+1}} \frac{1}{2\pi i}\int_{(\mu)} \frac{\Gamma(\frac{a}{2}+1-s)\Gamma(-\frac{a}{4}+ \frac{s}{2}) \Gamma(\frac{a}{4}+ \frac{s}{2}) \Gamma(\frac{1}{2} + \frac{a}{4}+ \frac{s}{2})}{\Gamma(\frac{1}{2} - \frac{a}{4} - \frac{s}{2}) \Gamma(1 - \frac{a}{4}- \frac{s}{2}) \Gamma(\frac{1}{2} + \frac{a}{4}- \frac{s}{2})} \left(\frac{\pi^3ny}{D_\K} \right)^{-s} ds.
\end{align}
We apply \eqref{Duplication formula} on the first gamma factor in the numerator and replace $s$ by $-2s$ in \eqref{Eqn:IKan} to deduce the above integral as
\begin{align*}
I_{\K, a}(n) = \frac{2^{a/2}}{\sqrt{\pi}y^{a/2+1}} \frac{1}{2\pi i} \int_{(-\frac{\mu}{2})} \frac{\Gamma(1+\frac{a}{4}+s) \Gamma(-\frac{a}{4}-s) \Gamma(\frac{a}{4}-s) \Gamma(\frac{1}{2}+\frac{a}{4}-s)}{\Gamma(\frac{1}{2}-\frac{a}{4}+s) \Gamma(1-\frac{a}{4}+s)} \left(\frac{4\pi^6n^2y^2}{D_\K^2} \right)^{s} \, ds.
\end{align*}
Thus \eqref{Inverse Slater} readily implies that
\begin{align*}
I_{\K, a}(n) &= \frac{2\pi^{\frac{1-3a}{2}}D_{\K}^{a/2}}{n^{a/2}y \sin(\pi a)}\Bigg[ \frac{2^{-2a}}{\Gamma^2(1-a)}\pFq14{1}{1-\frac{a}{2},1-\frac{a}{2},\frac{1-a}{2},\frac{1-a}{2}}{-\frac{4\pi^{6}n^2}{y^2D_{\mathbb{K}}^2}} \nonumber\\
&- \left(\frac{4\pi^{6}n^2}{y^2D_{\mathbb{K}}^2} \right)^{a/2} \left\{ \cos\left(\frac{\pi a}{2}\right) {\rm ber}\left(4\pi\sqrt{\frac{2n\pi}{yD_{\mathbb{K}}}} \right) -  \sin\left(\frac{\pi a}{2}\right) {\rm bei}\left(4\pi\sqrt{\frac{2n\pi}{yD_{\mathbb{K}}}} \right) \right\} \Bigg].
\end{align*}
Invoking the above evaluation of $I_{\K, a}(n)$ into \eqref{Apply Voronoi}, we conclude our result for $0 < \Re(a) < 1$ and $y > 0$. It remains to show next that the result is also valid for $\Re(a) > -1$ and $\Re(y) > 0$. The result in 
\cite[Corollary 7.119, p. 430]{Bordelles} implies that 
$\sigma_{\K, -a}(n) \leq \sum_{d \mid n} \sigma_0(d) d^{-a},$ where $\sigma_0(d)$ is the divisor function $d(n)$. Using the elementary bound of divisor function we can bound $\sigma_{\K, -a}(n)$ as
\begin{align}\label{Bound for sigma}
\sigma_{\K, -a}(n) \ll \begin{cases}
n^\epsilon & \text{ for } \Re(a)>0\\
n^{\epsilon - \Re(a)} & \text{ for } \Re(a)<0
\end{cases}
\end{align}
where $\epsilon$ is arbitrarily small positive quantity. We next employ the bounds from Lemma \ref{Asymptotic for summand} and \eqref{Bound for sigma} together to conclude that the series on the right hand side of \eqref{Eqn:Lambert series} converges uniformly as long as $\Re(a) > -1$. Since the summand of the series is analytic for $\Re(a) > -1$, by Weierstrass’ theorem on analytic functions, we see that it represents an analytic function of $a$ when $\Re(a) > -1$.

On the other hand, the left-hand side of \eqref{Eqn:Lambert series} is also analytic for $\Re(a) > -1$, hence by the principle of analytic continuation \eqref{Eqn:Lambert series} holds for $\Re(a) > -1$ and $y > 0$. The both sides of \eqref{Eqn:Lambert series} are seen to be analytic as a function of $y$, in $\Re(y) > 0$. Therefore the principle of analytic continuation concludes \eqref{Eqn:Lambert series} for $\Re(a) > -1$ and $\Re(y) > 0$.


In the following lemma we prove an identity, which is crucial in proving Theorem \ref{Analytic continuation}.
\begin{lemma}\label{1F4 lower reduction}
For any $a,\ z\in\C$, we have
\begin{align*}
{}_1F_4\left(\begin{matrix}
1\\ 1 -\frac{a}{2},  1 -\frac{a}{2},  \frac{1-a}{2}, \frac{1-a}{2} 
\end{matrix} \bigg | -z \right) &= \Gamma^2(1-a) \left[ \sum_{k=0}^{m-1} \frac{(-1)^{k}\left(16z\right)^{-k-1}}{\Gamma^2\left(-1-a-2k\right)} + \frac{(-z)^{-m}}{2^{4m}\Gamma^2(1-a-2m)}\right.\\
&\left.\quad\times{}_1F_4\left(\begin{matrix}
1\\ 1 -\frac{a}{2}-m,  1 -\frac{a}{2}-m,  \frac{1-a}{2}-m, \frac{1-a}{2}-m 
\end{matrix} \bigg | -z\right)\right].
\end{align*}
\end{lemma}

\begin{proof}
We use the following reduction formula repeatedly for ${}_1F_4$, which is given by
\begin{multline}\label{reduction ref}
\pFq14{a+1}{b_1+1,  b_2+1,  b_3+1,  b_4+1}  {x}=- \frac{b_1 b_2 b_3 b_4}{x} \left[ {}_1F_4\left( \begin{matrix}
a\\ b_1,  b_2,  b_3,  b_4 
\end{matrix} \bigg | x \right) - {}_1F_4\left( \begin{matrix}
a+1\\ b_1,  b_2,  b_3,  b_4 
\end{matrix} \bigg | x \right)\right].
\end{multline}
The above formula with $a=0$, $b_1 = b_2 = -\frac{a}{2}$, $b_3 = b_4 = -\frac{1+a}{2}$ and $x = -z$ provides
\begin{align*}
&{}_1F_4\left(\begin{matrix}
1\\ 1 -\frac{a}{2},  1 -\frac{a}{2},  \frac{1-a}{2}, \frac{1-a}{2} 
\end{matrix} \bigg | -z \right)=  \frac{\left(-\frac{a}{2}\right)^2 \left(-\frac{1+a}{2}\right)^2 }{z} \left[ 1 - {}_1F_4\left( \begin{matrix}
1\\ -\frac{a}{2}, -\frac{a}{2}, -\frac{1+a}{2}, -\frac{1+a}{2}
\end{matrix} \bigg | -z \right)\right]\\
&= \Gamma^2\left(1-\frac{a}{2}\right) \Gamma^2\left(\frac{1-a}{2}\right)\left[\frac{\left(z\right)^{-1}}{\Gamma^2\left(-\frac{a}{2}\right) \Gamma^2\left(-\frac{1+a}{2}\right)} - \frac{\left(z\right)^{-1}}{\Gamma^2\left(-\frac{a}{2}\right) \Gamma^2\left(-\frac{1+a}{2}\right)} {}_1F_4\left( \begin{matrix}
1\\ -\frac{a}{2}, -\frac{a}{2}, -\frac{1+a}{2}, -\frac{1+a}{2}
\end{matrix} \bigg | -z \right) \right].
\end{align*}
Applying \eqref{reduction ref} on the right-hand side of the above equation, we get
\begin{align*}
&{}_1F_4\left(\begin{matrix}
1\\ 1 -\frac{a}{2},  1 -\frac{a}{2},  \frac{1-a}{2}, \frac{1-a}{2} 
\end{matrix} \bigg | -z \right)\\
&= \Gamma^2\left(1-\frac{a}{2}\right) \Gamma^2\left(\frac{1-a}{2}\right) \left[\frac{\left(z\right)^{-1}}{\Gamma^2\left(-\frac{a}{2}\right) \Gamma^2\left(-\frac{1+a}{2}\right)} - \frac{\left(z\right)^{-2}}{\Gamma^2\left(-\frac{a}{2}-1\right) \Gamma^2\left(-\frac{3+a}{2}\right)}\right.\\
&\qquad \qquad \qquad \qquad
\left. + \frac{\left(z\right)^{-2}}{\Gamma^2\left(-\frac{a}{2}-1\right) \Gamma^2\left(-\frac{3+a}{2}\right)}{}_1F_4\left( \begin{matrix}
1\\ -\frac{a}{2}-1, -\frac{a}{2}-1, -\frac{3+a}{2}, -\frac{3+a}{2}
\end{matrix} \bigg | -z \right)\right].
\end{align*}
We repeat this process $m$-times and obtain
\begin{align*}
{}_1F_4&\left(\begin{matrix}
1\\ 1 -\frac{a}{2},  1 -\frac{a}{2},  \frac{1-a}{2}, \frac{1-a}{2} 
\end{matrix} \bigg | -z \right) 
= \Gamma^2\left(1-\frac{a}{2}\right) \Gamma^2\left(\frac{1-a}{2}\right) \left[ \sum_{j=1}^{m} (-1)^{j-1}  \frac{\left(z\right)^{-j}}{\Gamma^2\left(1-\frac{a}{2}-j\right) \Gamma^2\left(\frac{1-a}{2}-j\right)}\right.\\
&\left. + (-1)^m \frac{\left(z\right)^{-m}}{\Gamma^2\left(1-\frac{a}{2}-m\right) \Gamma^2\left(\frac{1-a}{2}-m\right)} {}_1F_4\left(\begin{matrix}
1\\ 1 -\frac{a}{2}-m,  1 -\frac{a}{2}-m,  \frac{1-a}{2}-m, \frac{1-a}{2}-m 
\end{matrix} \bigg | -z \right)\right].
\end{align*}
Substituting $j=k-1$ in the finite sum and applying \eqref{Duplication formula}, we conclude our lemma.
\end{proof}

\subsection{Proof of Theorem \ref{Analytic continuation}}
For $\Re(a)>-1$, we rewrite the identity \eqref{Eqn:Lambert series} as
\begin{align}\label{before finite sum simplification}
&\sum_{n=1}^\infty\sigma_{\mathbb{K},a}(n)e^{-ny}+\frac{1}{2}\zeta_{\mathbb{K}}(-a)-\frac{\zeta_{\mathbb{K}}(1-a)}{y}-\frac{2\pi h}{w\sqrt{D_\K}}\frac{\Gamma(a+1)\zeta(a+1)}{y^{a+1}}  
\nonumber\\
& =\frac{4\pi^{2-2a}D_\K^{a-\frac{1}{2}}}{y\sin(\pi a)} \sum_{n=1}^\infty \sigma_{\mathbb{K},-a}(n)\Bigg\{\frac{2^{-2a}}{\Gamma^2(1-a)}\pFq14{1}{1-\frac{a}{2},1-\frac{a}{2},\frac{1-a}{2},\frac{1-a}{2}}{-\frac{4\pi^6n^2}{y^2D_{\K}^2}}\nonumber\\
&-\left(\frac{4\pi^6n^2}{y^2D_{\K}^2}\right)^{\frac{a}{2}}\bigg(\cos\left(\frac{\pi a}{2}\right) {\rm ber}\left(4\pi\sqrt{\frac{2n\pi}{yD_{\mathbb{K}}}} \right)-\sin\left(\frac{\pi a}{2}\right){\rm bei}\left(4\pi\sqrt{\frac{2n\pi}{yD_{\mathbb{K}}}} \right)\bigg)
\nonumber\\
&-\frac{1}{2^{2a}} \sum_{k=0}^m \frac{(-1)^k \left(\frac{64\pi^6n^2}{y^2D_{\K}^2}\right)^{-k-1}}{\Gamma^2(-1-a-2k)}\Bigg\}+\frac{(2\pi)^{2-2a}D_\K^{a-\frac{1}{2}}}{y\sin(\pi a)} \sum_{n=1}^\infty \sigma_{\mathbb{K},-a}(n)\sum_{k=0}^m \frac{(-1)^k \left(\frac{64\pi^6n^2}{y^2D_{\K}^2}\right)^{-k-1}}{\Gamma^2(-1-a-2k)}.
\end{align}
The last term of the above expression can be simplified using \eqref{Dirichlet series of zeta zetak} as
\begin{align*}
\sum_{k=0}^m\frac{(-1)^k\left(\frac{64\pi^6}{y^2D_{\K}^2}\right)^{-k-1}}{\Gamma^2(-1-a-2k)}\sum_{n=1}^\infty \frac{\sigma_{\mathbb{K},-a}(n)}{n^{2k+2}}=\sum_{k=0}^m\frac{(-1)^k\left(\frac{64\pi^6}{y^2D_{\K}^2}\right)^{-k-1}}{\Gamma^2(-1-a-2k)}\zeta(2k+2)\zeta_{\K}(2k+a+2).
\end{align*}
Therefore for $Re(a)>-1$, \eqref{before finite sum simplification} can be written as
\begin{align}\label{before ac}
&\sum_{n=1}^\infty\sigma_{\mathbb{K},a}(n)e^{-ny}+\frac{1}{2}\zeta_{\mathbb{K}}(-a)-\frac{\zeta_{\mathbb{K}}(1-a)}{y}-\frac{2\pi h}{w\sqrt{D_\K}}\frac{\Gamma(a+1)\zeta(a+1)}{y^{a+1}}  
\nonumber\\
& =\frac{4\pi^{2-2a}D_\K^{a-\frac{1}{2}}}{y\sin(\pi a)} \sum_{n=1}^\infty \sigma_{\mathbb{K},-a}(n)\Bigg\{\frac{2^{-2a}}{\Gamma^2(1-a)}\pFq14{1}{1-\frac{a}{2},1-\frac{a}{2},\frac{1-a}{2},\frac{1-a}{2}}{-\frac{4\pi^6n^2}{y^2D_{\K}^2}}\nonumber\\
&-\left(\frac{4\pi^6n^2}{y^2D_{\K}^2}\right)^{\frac{a}{2}}\bigg(\cos\left(\frac{\pi a}{2}\right) {\rm ber}\left(4\pi\sqrt{\frac{2n\pi}{yD_{\mathbb{K}}}} \right)-\sin\left(\frac{\pi a}{2}\right){\rm bei}\left(4\pi\sqrt{\frac{2n\pi}{yD_{\mathbb{K}}}} \right)\bigg)
\nonumber\\
&-\frac{1}{2^{2a}} \sum_{k=0}^m \frac{(-1)^k \left(\frac{64\pi^6n^2}{y^2D_{\K}^2}\right)^{-k-1}}{\Gamma^2(-1-a-2k)}\Bigg\} +\frac{yD_\K^{a+\frac{3}{2}}}{(2\pi)^{2a+4}\sin(\pi a)}\sum_{k=0}^m\frac{(-1)^k\zeta(2k+2)\zeta_{\mathbb{K}}(2k+a+2)}{\Gamma^2(-a-1-2k)}\left(\frac{8\pi^{3}}{yD_{\mathbb{K}}}\right)^{-2k}.
\end{align}
Invoking Lemma \ref{Asymptotic for summand} and \eqref{Bound for sigma}, we have
\begin{align*}
\sigma_{\mathbb{K},-a}(n)&\Bigg\{\frac{2^{-2a}}{\Gamma^2(1-a)}\pFq14{1}{1-\frac{a}{2},1-\frac{a}{2},\frac{1-a}{2},\frac{1-a}{2}}{-\frac{4\pi^6n^2}{y^2D_{\K}^2}}-\left(\frac{4\pi^6n^2}{y^2D_{\K}^2}\right)^{\frac{a}{2}}\bigg(\cos\left(\frac{\pi a}{2}\right) {\rm ber}\left(4\pi\sqrt{\frac{2n\pi}{yD_{\mathbb{K}}}} \right)\nonumber\\
&-\sin\left(\frac{\pi a}{2}\right){\rm bei}\left(4\pi\sqrt{\frac{2n\pi}{yD_{\mathbb{K}}}} \right)\bigg)-\frac{1}{2^{2a}} \sum_{k=0}^m \frac{(-1)^k \left(\frac{64\pi^6n^2}{y^2D_{\K}^2}\right)^{-k-1}}{\Gamma^2(-1-a-2k)}\Bigg\}\nonumber\\
&\ll\begin{cases}
n^{-2m-4+\epsilon}& \text{ for }\ \Re(a)\geq0\\
n^{-2m-4-\Re(a)+\epsilon} &\text{ for }\ \Re(a)<0.
\end{cases}
\end{align*}
This shows that the infinite series on the right-hand side of \eqref{before ac} is uniformly convergent for $\Re(a)\geq-2m-3+\epsilon$ where $\epsilon>0$. Since the summand of the above series is analytic for $\Re(a) >-2m - 3$, it follows from Weierstrass’ theorem that this series represents an analytic function of $a$ for $\Re(a) >-2m-3$.

The left-hand side of \eqref{before ac} as well as the finite sum on its right-hand side are also analytic for $\Re(a) > -2m-3$, hence by the principle of analytic continuation, \eqref{before ac} holds for $\Re(a) > -2m-3$ with $m \geq 0$. Finally we apply Lemma \ref{1F4 lower reduction} in \eqref{before ac} to conclude our theorem.
\qed


The following lemma is crucial in proving the next results and seems new in the literature.
\begin{lemma}\label{der1f4@2m-1}
Let $\ell\in\mathbb{Z}$. Then for $z>0$, we have
\begin{align}\label{derivative1}
&\frac{d}{da}\pFq14{1}{1-\frac{a}{2}+\ell,1-\frac{a}{2}+\ell,\frac{1-a}{2}+\ell,\frac{1-a}{2}+\ell}{-z^4}\Bigg|_{a=2\ell-1}\nonumber\\
&=\frac{1}{2z^2}\left\{(\gamma-1+\log(2z))\mathrm{bei}\left(4z\right)+\frac{\pi}{4}\mathrm{ber}\left(4z\right)+\mathrm{kei}\left(4z\right)\right\}
\end{align}
and 
\begin{align}\label{derivative2}
&\frac{d}{da}\pFq14{1}{1-\frac{a}{2}+\ell,1-\frac{a}{2}+\ell,\frac{1-a}{2}+\ell,\frac{1-a}{2}+\ell}{-z^4}\Bigg|_{a=2\ell}\nonumber\\
&=2(\gamma+\log(2z))\mathrm{ber}(4z)-\frac{\pi}{2}\mathrm{bei}(4z)+2\mathrm{ker}(4z).
\end{align}
\end{lemma}
\begin{proof}
The series definition of ${}_1F_4$ and \eqref{Duplication formula} yields
\begin{align*}
\pFq14{1}{1-\frac{a}{2}+\ell,1-\frac{a}{2}+\ell,\frac{1-a}{2}+\ell,\frac{1-a}{2}+\ell}{-z^4}&=\sum_{k=0}^\infty \frac{1}{\left(\frac{1-a}{2}+\ell\right)_k^2\left(1-\frac{a}{2}+\ell\right)_k^2}(-z^4)^k\nonumber\\
&=\sum_{k=0}^\infty\frac{(-1)^k\Gamma^2(1-a+2\ell)}{\Gamma^2(1-a+2\ell+2k)}(2z)^{4k}.
\end{align*}
Differentiating both sides of the above equation with respect to $a$, we obtain
\begin{align*}
\frac{d}{da}\pFq14{1}{1-\frac{a}{2}+\ell,1-\frac{a}{2}+\ell,\frac{1-a}{2}+\ell,\frac{1-a}{2}+\ell}{-z^4}&=-2\sum_{k=0}^\infty \frac{(-1)^k\Gamma^2(1-a+2\ell)\psi(1-a+2\ell)}{\Gamma^2(1-a+2\ell+2k)}(2z)^{4k}\nonumber\\
&+2\sum_{k=0}^\infty \frac{(-1)^k\Gamma^2(1-a+2\ell)\psi(1-a+2\ell+2k)}{\Gamma^2(1-a+2\ell+2k)}(2z)^{4k}.
\end{align*} 
Thus the above expression at $a=2\ell-1$ becomes
\begin{align}\label{brybefore}
\frac{d}{da}\pFq14{1}{1-\frac{a}{2}+\ell,1-\frac{a}{2}+\ell,\frac{1-a}{2}+\ell,\frac{1-a}{2}+\ell}{-z^4}\Bigg|_{a=2\ell-1}&=2(\gamma-1)\sum_{k=0}^\infty \frac{(-1)^k(2z)^{4k}}{\Gamma^2(2k+2)}\nonumber\\
&+2\sum_{k=0}^\infty \frac{(-1)^k\psi(2k+2)(2z)^{4k}}{\Gamma^2(2k+2)}.
\end{align}
Applying \eqref{BerBei}, the first infinite series on the right hand side of the above equation reduces to 
\begin{align}\label{series0f3ber}
\sum_{k=0}^\infty \frac{(-1)^k(2z)^{4k}}{\Gamma^2(2k+2)}=\pFq03{-}{\frac{3}{2},\frac{3}{2},1}{-z^4}=\frac{1}{4z^2}\mathrm{bei}(4z),
\end{align} 
It follows from \cite[p.~439, Formula~78]{bry} that for $t>0$, we have
\begin{align}\label{bry2k+2}
\sum_{k=0}^\infty\frac{(-1)^k\psi(2k+2)}{\Gamma^2(2k+2)}t^{4k}=\frac{1}{4t^2}\left\{\pi\mathrm{ber}(2t)+4\log(t)\mathrm{bei}(2t)+4\mathrm{kei}(2t)\right\}.
\end{align}
Thus \eqref{bry2k+2} with $t=2z$ and \eqref{series0f3ber} together implies \eqref{derivative1}. We can also obtain \eqref{derivative2} in a similar way by substituting $a$ by $2\ell$ in \eqref{brybefore} and applying the following relation
\begin{align*}
\sum_{k=0}^\infty \frac{\psi(2k+1)(-1)^kt^{4k}}{((2k)!)^2}=\log(t)\mathrm{ber}(2t)-\frac{\pi}{4}\mathrm{bei}(2t)+\mathrm{ker}(2t)
\end{align*}
for $t>0$, which follows from \cite[p.~439, Formula 77]{bry}.
\end{proof}

\subsection{Proof of Theorem \ref{a=-1 case}}
The main idea here is to take  limit on  both sides of the identity in Theorem \ref{Analytic continuation} as $a\to-1$. Therefore, it is sufficient to consider $m=0$ in Theorem \ref{Analytic continuation}. Applying the Laurent series expansion at $s=1$ of the functions $\zeta_{\K}(s)$, $\zeta(s)$ and the expansion of $\Gamma(s)$ at $s=0$, the following limit reduces as 
\begin{align}\label{lima=-1-1}
\lim_{a\to-1}\left\{\frac{1}{2}\zeta_{\mathbb{K}}(-a)-\frac{2\pi h}{w\sqrt{D_\K}}\frac{\Gamma(a+1)\zeta(a+1)}{y^{a+1}}\right\}=\frac{1}{2}\left\{L'(1,\chi_D)+L(1,\chi_D)\left(2\gamma-\log\left(\frac{2\pi}{y}\right)\right)\right\}.
\end{align}
We next evaluate the following limit
\begin{align}\label{lima=-1}
L_{-1}:=&\lim_{a\to-1}\frac{1}{\sin(\pi a)}\Bigg[\frac{2^{-2a}}{\Gamma^2(1-a)}
\bigg \{ {}_1F_4\left(\begin{matrix}
1\\ 1 -\frac{a}{2},  1 -\frac{a}{2},  \frac{1-a}{2}, \frac{1-a}{2}
\end{matrix} \bigg | -\frac{4\pi^6n^2}{y^2D_\K^2} \right) - a^2(a+1)^2\left(\frac{64\pi^6n^2}{y^2D_\K^2}\right)^{-1} \bigg\}\nonumber\\
&-\left( \frac{4\pi^6n^2}{y^2D_\K^2} \right)^{\frac{a}{2}}\bigg \{ \cos\left(\frac{\pi a}{2}\right) {\rm ber}\left(4\pi\sqrt{\frac{2n\pi}{yD_{\mathbb{K}}}} \right)
-\sin\left(\frac{\pi a}{2}\right){\rm bei}\left(4\pi\sqrt{\frac{2n\pi}{yD_{\mathbb{K}}}} \right)\bigg \}\Bigg].
\end{align}
For $a=-1$, it follows from \eqref{BerBei} that
\begin{align*}
\frac{2^{-2a}}{\Gamma^2(1-a)}&
\bigg \{ {}_1F_4\left(\begin{matrix}
1\\ 1 -\frac{a}{2},  1 -\frac{a}{2},  \frac{1-a}{2}, \frac{1-a}{2}
\end{matrix} \bigg | -\frac{4\pi^6n^2}{y^2D_\K^2} \right)- a^2(a+1)^2\left(\frac{64\pi^6n^2}{y^2D_\K^2}\right)^{-1} \bigg\}\\
&=\left( \frac{4\pi^6n^2}{y^2D_\K^2} \right)^{\frac{a}{2}}\bigg \{ \cos\left(\frac{\pi a}{2}\right) {\rm ber}\left(4\pi\sqrt{\frac{2n\pi}{yD_{\mathbb{K}}}} \right)-\sin\left(\frac{\pi a}{2}\right){\rm bei}\left(4\pi\sqrt{\frac{2n\pi}{yD_{\mathbb{K}}}} \right)\bigg \}
\end{align*}
Thus we have $0/0$ form in the limit \eqref{lima=-1}. Applying L'Hopital's rule and \eqref{derivative1} with $\ell=0$, the limit evaluates as
\begin{align}\label{lima=-12}
L_{-1}=-\frac{yD_\K}{n\pi^4}\mathrm{ker}\left(4\pi\sqrt{\frac{2n\pi}{yD_{\K}}}\right).
\end{align}
We also determine the following limit as
\begin{align}\label{lima=-13}
\lim_{a\to-1}\frac{\zeta_\K(a+2)}{\sin(\pi a)\Gamma^2(-a-1)}&=\lim_{s\to1}\frac{\zeta_\K(s)}{\sin(\pi s)\Gamma^2(1-s)}=-\frac{1}{\pi}L(1,\chi_D).
\end{align}
Finally, taking  limit as $a\to-1$ on the both sides of \eqref{Eqn:Analytic continuation} and applying  \eqref{lima=-1-1}, \eqref{lima=-12} and \eqref{lima=-13} together we conclude our theorem.

\subsection{Proof of Corollary \ref{a=0 case}}
In this case, we need to take limit on  both sides of the identity in Theorem \ref{Lambert series} as $a\to0$. It follows by change of variable that
\begin{align*}
&\lim_{a\to0}\left\{\frac{\zeta_{\mathbb{K}}(1-a)}{y}+\frac{L(1, \chi_D)\Gamma(a+1)\zeta(a+1)}{y^{a+1}}\right\}=\frac{1}{y}\lim_{s\to1}\left\{\zeta_{\mathbb{K}}(s)+\frac{L(1, \chi_D)\Gamma(2-s)\zeta(2-s)}{y^{1-s}}\right\}.
\end{align*}
Inserting the Laurent series expansion at $s=1$ of the functions $\zeta_{\K}(s)$, $\zeta(s)$ and $\Gamma(s)$, the above limit reduces to 
\begin{align}\label{Hlim}
\lim_{a\to0}\left\{\frac{\zeta_{\mathbb{K}}(1-a)}{y}+\frac{L(1, \chi_D)\Gamma(a+1)\zeta(a+1)}{y^{a+1}}\right\}
&=\frac{L'(1, \chi_D)+L(1, \chi_D) (\gamma - \log(y))}{y}.
\end{align}
We also have singularities in the first and second term of the summand on the right-hand side of \eqref{Eqn:Lambert series}. Therefore, we need to evaluate the following limit
\begin{align}\label{oby0}
L_0&:=\lim_{a\to0}\frac{1}{a}\Bigg\{\frac{2^{-2a}\Gamma(a+1)}{\Gamma(1-a)}\left(\frac{4\pi^6n^2}{y^2D_{\mathbb{K}}^2}\right)^{-\frac{a}{4}}\pFq14{1}{1-\frac{a}{2},1-\frac{a}{2},\frac{1-a}{2},\frac{1-a}{2}}{-\frac{4\pi^6n^2}{y^2D_{\mathbb{K}}^2}}\nonumber\\
&\quad-\Gamma\left(1-\frac{a}{2}\right)\Gamma\left(1+\frac{a}{2}\right)\left(\frac{4\pi^6n^2}{y^2D_{\mathbb{K}}^2}\right)^{\frac{a}{4}}\mathrm{ber}\left(4\pi\sqrt{\frac{2n\pi}{yD_{\K}}}\right)\Bigg\}
\end{align}
where we have applied \eqref{Reflection formula} on the gamma factors. It can be noted that for $a=0$, we have
\begin{align*}
&\frac{2^{-2a}\Gamma(a+1)}{\Gamma(1-a)}\left(\frac{4\pi^6n^2}{y^2D_{\mathbb{K}}^2}\right)^{-\frac{a}{4}}\pFq14{1}{1-\frac{a}{2},1-\frac{a}{2},\frac{1-a}{2},\frac{1-a}{2}}{-\frac{4\pi^6n^2}{y^2D_{\mathbb{K}}^2}}\nonumber\\
&=\Gamma\left(1-\frac{a}{2}\right)\Gamma\left(1+\frac{a}{2}\right)\left(\frac{4\pi^6n^2}{y^2D_{\mathbb{K}}^2}\right)^{\frac{a}{4}}\mathrm{ber}\left(4\pi\sqrt{\frac{2n\pi}{yD_{\K}}}\right),
\end{align*}
thus we have $0/0$ form in \eqref{oby0}. Applying L'Hopital's rule, we evaluate the limit as
\begin{align*}
L_0&=\left(-2\gamma-\log\left(\frac{8\pi^3n}{yD_{\mathbb{K}}}\right)\right)\mathrm{ber}\left(4\pi\sqrt{\frac{2n\pi}{yD_{\K}}}\right)+\frac{d}{da}\pFq14{1}{1-\frac{a}{2},1-\frac{a}{2},\frac{1-a}{2},\frac{1-a}{2}}{-\frac{4\pi^6n^2}{y^2D_{\mathbb{K}}^2}}\Bigg|_{a=0}.
\end{align*}
We next employ \eqref{derivative2} with $\ell=0$ in the above equation to get
\begin{align}\label{l0result}
L_0
=2\mathrm{ker}\left(4\pi\sqrt{\frac{2n\pi}{yD_{\K}}}\right)-\frac{\pi}{2}\mathrm{bei}\left(4\pi\sqrt{\frac{2n\pi}{yD_{\K}}}\right).
\end{align}
Finally, taking  limit as $a\to0$ on the both sides of \eqref{Eqn:Lambert series} and applying  $\zeta_{\K}(0)=-h/w$,  \eqref{Hlim} and \eqref{l0result} together  conclude our corollary.

\section{Analogues of transformation formulas for Eisenstein series}
In this section, we mainly study the infinite series associated to $\sigma_{\K, a}(n)$, which is analogous to Eisenstein series. 

\subsection{Proof of Theorem \ref{a=2m-1,2m}}
We first prove the identity \eqref{a=2m-1} and for that we need to take the limit for $a\to2m-1$ on the both sides of \eqref{Eqn:Lambert series}. Substituting $a$ by $a-2m$ in Lemma \ref{1F4 lower reduction}, we apply it on the right-hand side of \eqref{Eqn:Lambert series} to obtain
\begin{align}\label{expanded form}
&\frac{2^{1+\frac{a}{2}}\pi^{1-\frac{a}{2}}D_{\K}^{\frac{a-1}{2}}}{y^{1+\frac{a}{2}}}\sum_{n=1}^\infty\frac{\sigma_{\K,-a}(n)}{n^{-\frac{a}{2}}}\Bigg[\frac{\Gamma\left(\frac{a}{2}\right)\Gamma\left(\frac{1+a}{2}\right)}{\Gamma\left(1-\frac{a}{2}\right)\Gamma\left(\frac{1-a}{2}\right)}\left(\frac{4\pi^6n^2}{y^2D_{\mathbb{K}}^2}\right)^{-\frac{a}{4}}-\Gamma\left(\frac{a}{2}\right)\Gamma\left(\frac{1+a}{2}\right)\Gamma\left(1-\frac{a}{2}\right)\Gamma\left(\frac{1-a}{2}\right)\nonumber\\
&\times\left(\frac{4\pi^6n^2}{y^2D_{\mathbb{K}}^2}\right)^{-\frac{a}{4}}\Bigg\{\sum_{k=0}^{m-2}\frac{(-1)^k}{\Gamma\left(2-\frac{a}{2}+k\right)^2\Gamma\left(\frac{3-a}{2}+k\right)^2}\left(\frac{4\pi^6n^2}{y^2D_{\mathbb{K}}^2}\right)^{k+1}+\frac{(-1)^{m-1}}{\Gamma\left(1-\frac{a}{2}+m\right)^2\Gamma\left(\frac{1-a}{2}+m\right)^2}\nonumber\\
&\times\left(\frac{4\pi^6n^2}{y^2D_\mathbb{{\mathbb{K}}}^2}\right)^m \pFq14{1}{1-\frac{a}{2}+m,1-\frac{a}{2}+m,\frac{1-a}{2}+m,\frac{1-a}{2}+m}{-\frac{4\pi^6n^2}{y^2D_{\mathbb{K}}^2}}\Bigg\}\nonumber\\
&-\left( \frac{4\pi^6n^2}{y^2D_\K^2} \right)^{\frac{a}{2}}\bigg(\cos\left(\frac{\pi a}{2}\right) {\rm ber}\left(4\pi\sqrt{\frac{2n\pi}{yD_{\mathbb{K}}}} \right)-\sin\left(\frac{\pi a}{2}\right){\rm bei}\left(4\pi\sqrt{\frac{2n\pi}{yD_{\mathbb{K}}}} \right)\bigg)\Bigg].
\end{align}
Firstly, we evaluate the limit
\begin{align}\label{gammaequibefore}
L_{2m-1}:&=\lim_{a\to 2m-1}\frac{1}{(a-2m+1)}\Bigg\{(a-2m+1)\Gamma\left(\frac{a}{2}\right)\Gamma\left(\frac{1+a}{2}\right)\Gamma\left(1-\frac{a}{2}\right)\Gamma\left(\frac{1-a}{2}\right)\left(\frac{4\pi^6n^2}{y^2D_{\mathbb{K}}^2}\right)^{m-\frac{a}{4}}\nonumber\\
&\times\frac{(-1)^{m}}{\Gamma\left(1-\frac{a}{2}+m\right)^2\Gamma\left(\frac{1-a}{2}+m\right)^2}\pFq14{1}{1-\frac{a}{2}+m,1-\frac{a}{2}+m,\frac{1-a}{2}+m,\frac{1-a}{2}+m}{-\frac{4\pi^6n^2}{y^2D_{\mathbb{K}}^2}}\nonumber\\
&-(a-2m+1)\Gamma\left(-\frac{1}{2}-\frac{a}{2}\right)\Gamma\left(\frac{3+a}{2}\right)\left(\frac{4\pi^6n^2}{y^2D_{\mathbb{K}}^2}\right)^{\frac{a}{4}}{\rm bei}\left(4\pi\sqrt{\frac{2n\pi}{yD_{\mathbb{K}}}} \right)\Bigg\}.
\end{align}
Using \eqref{Reflection formula}, the following gamma factors can be reduced as
\begin{align}\label{gammaequi1}
(a-2m+1)\Gamma\left(\frac{1-a}{2}\right)\Gamma\left(\frac{1+a}{2}\right)=2(-1)^{m}\Gamma\left(\frac{a-2m+1}{2}+1\right)\Gamma\left(\frac{1-a+2m}{2}\right)
\end{align}
and 
\begin{align}\label{gammaequi2}
(a-2m+1)\Gamma\left(-\frac{1}{2}-\frac{a}{2}\right)\Gamma\left(\frac{3+a}{2}\right)&=2(-1)^{m+1}\Gamma\left(\frac{a-2m+1}{2}+1\right)\Gamma\left(\frac{1-a+2m}{2}\right).
\end{align}
We then plug back \eqref{gammaequi1} and \eqref{gammaequi2} into \eqref{gammaequibefore} to obtain
\begin{align}\label{lhopitalbefore}
L_{2m-1}&=\lim_{a\to 2m-1}\frac{1}{(a-2m+1)}\Bigg\{\frac{2\Gamma\left(\frac{a}{2}\right)\Gamma\left(1-\frac{a}{2}\right)\Gamma\left(\frac{a-2m+1}{2}+1\right)}{\Gamma\left(1-\frac{a}{2}+m\right)^2\Gamma\left(\frac{1-a}{2}+m\right)}\left(\frac{4\pi^6n^2}{y^2D_{\mathbb{K}}^2}\right)^{m-\frac{a}{4}}\nonumber\\
&\quad\times\pFq14{1}{1-\frac{a}{2}+m,1-\frac{a}{2}+m,\frac{1-a}{2}+m,\frac{1-a}{2}+m}{-\frac{4\pi^6n^2}{y^2D_{\mathbb{K}}^2}}\nonumber\\
&\quad+2(-1)^m\Gamma\left(\frac{a-2m+1}{2}+1\right)\Gamma\left(\frac{1-a+2m}{2}\right)\left(\frac{4\pi^6n^2}{y^2D_{\mathbb{K}}^2}\right)^{\frac{a}{4}}{\rm bei}\left(4\pi\sqrt{\frac{2n\pi}{yD_{\mathbb{K}}}}\right)\Bigg\}.
\end{align}
Applying \eqref{BerBei} it is easy to see that for $a=2m-1$,  
we have $0/0$ form in \eqref{lhopitalbefore}, hence we use L'Hopital's rule to evaluate the limit as
\begin{align}\label{conclude}
L_{2m-1}
&=(-1)^m\left(\frac{2\pi^3n}{yD_{\K}}\right)^{m-\frac{1}{2}}\left(4(\gamma-1)+\log\left(\frac{64\pi^6n^2}{y^2D_{\K}^2}\right)\right)\mathrm{bei}\left(4\pi\sqrt{\frac{2n\pi}{yD_{\mathbb{K}}}}\right)-8(-1)^m\left(\frac{4\pi^6n^2}{y^2D_{\mathbb{K}}^2}\right)^{\frac{m}{2}+\frac{1}{4}}\nonumber\\
&\quad\times\frac{d}{da}\pFq14{1}{1-\frac{a}{2}+m,1-\frac{a}{2}+m,\frac{1-a}{2}+m,\frac{1-a}{2}+m}{-\frac{4\pi^6n^2}{y^2D_{\mathbb{K}}^2}}\Bigg|_{a=2m-1}.
\end{align}
Substituting  \eqref{derivative1} into \eqref{conclude}, we obtain
\begin{align}\label{conclude2}
L_{2m-1}=(-1)^{m+1}2^{m-\frac{1}{2}}\pi^{3m-\frac{3}{2}}\left(\frac{n}{yD_{\mathbb{K}}}\right)^{m-\frac{1}{2}}\left\{\pi\mathrm{ber}\left(4\pi\sqrt{\frac{2n\pi}{yD_{\mathbb{K}}}}\right)+4\mathrm{kei}\left(4\pi\sqrt{\frac{2n\pi}{yD_{\mathbb{K}}}}\right)\right\}.
\end{align}
It is straightforward to see that
\begin{align}\label{effortless0}
\lim_{a\to 2m-1}&\Bigg\{\frac{\Gamma\left(\frac{a}{2}\right)\Gamma\left(\frac{1+a}{2}\right)}{\Gamma\left(1-\frac{a}{2}\right)\Gamma\left(\frac{1-a}{2}\right)}\left(\frac{4\pi^6n^2}{y^2D_{\mathbb{K}}^2}\right)^{-\frac{a}{4}}-\Gamma\left(\frac{a}{2}\right)\Gamma\left(\frac{1+a}{2}\right)\Gamma\left(1-\frac{a}{2}\right)\Gamma\left(\frac{1-a}{2}\right)\left(\frac{4\pi^6n^2}{y^2D_{\mathbb{K}}^2}\right)^{-\frac{a}{4}}\nonumber\\
&\times\sum_{k=0}^{m-2}\frac{(-1)^k}{\Gamma\left(2-\frac{a}{2}+k\right)^2\Gamma\left(\frac{3-a}{2}+k\right)^2}\left(\frac{4\pi^6n^2}{y^2D_{\mathbb{K}}^2}\right)^{k+1}\Bigg\}=0.
\end{align}
Thus taking limit as $a\to2m-1$ on the both sides of \eqref{Eqn:Lambert series} and using the fact that $\zeta_{\K}(s)$ has zeros on the negative integers, \eqref{expanded form}, \eqref{conclude2} and \eqref{effortless0} together yield
\begin{align}\label{a2m-1 in y form}
\sum_{n=1}^\infty\sigma_{\mathbb{K},2m-1}(n)e^{-ny}&=\frac{L(1,\chi_D)\Gamma(2m)\zeta(2m)}{y^{2m}}+\frac{4(-1)^{m+1}}{\sqrt{D_{\mathbb{K}}}}\left(\frac{2\pi}{y}\right)^{2m}\sum_{n=1}^\infty \frac{\sigma_{\mathbb{K},1-2m}(n)}{n^{1-2m}}\mathrm{kei}\left(4\pi\sqrt{\frac{2n\pi}{yD_{\mathbb{K}}}}\right).
\end{align}
Finally, we substitute $y$ by $\frac{8\pi^2\alpha}{D_{\K}}$ with $\alpha\beta=\frac{D_{\K}^2}{16\pi^2}$ in \eqref{a2m-1 in y form} to conclude \eqref{a=2m-1}.


We next show the second part of our theorem. The idea of the proof goes along the similar direction as in the previous part by taking limit as $a\to 2m$ on the both sides of \eqref{Eqn:Lambert series}. It follows from \eqref{expanded form} that the following limit needs to be evaluated:
\begin{align*}
L_{2m}:=&\lim_{a\to 2m}\frac{1}{(a-2m)}\Bigg\{(a-2m)\Gamma\left(\frac{a}{2}\right)\Gamma\left(\frac{1+a}{2}\right)\Gamma\left(1-\frac{a}{2}\right)\Gamma\left(\frac{1-a}{2}\right)\left(\frac{4\pi^6n^2}{y^2D_{\mathbb{K}}^2}\right)^{m-\frac{a}{4}}\nonumber\\
&\times\frac{(-1)^{m}}{\Gamma\left(1-\frac{a}{2}+m\right)^2\Gamma\left(\frac{1-a}{2}+m\right)^2}\pFq14{1}{1-\frac{a}{2}+m,1-\frac{a}{2}+m,\frac{1-a}{2}+m,\frac{1-a}{2}+m}{-\frac{4\pi^6n^2}{y^2D_{\mathbb{K}}^2}}\nonumber\\
&+(a-2m)\Gamma\left(-\frac{a}{2}\right)\Gamma\left(1+\frac{a}{2}\right)\left(\frac{4\pi^6n^2}{y^2D_{\mathbb{K}}^2}\right)^{\frac{a}{4}}\mathrm{ber}\left(4\pi\sqrt{\frac{2n\pi}{yD_{\mathbb{K}}}}\right)\Bigg\}.
\end{align*}
Applying \eqref{Reflection formula}, we can write the above limit as 
\begin{align*}
L_{2m}=\lim_{a\to 2m}&\frac{1}{(a-2m)}\Bigg\{\frac{2\Gamma\left(1+\frac{a}{2}-m\right)\Gamma\left(\frac{1+a}{2}\right)\Gamma\left(\frac{1-a}{2}\right)}{\Gamma\left(1-\frac{a}{2}+m\right)\Gamma\left(\frac{1-a}{2}+m\right)^2}\left(\frac{4\pi^6n^2}{y^2D_{\mathbb{K}}^2}\right)^{m-\frac{a}{4}}\nonumber\\
&\times\pFq14{1}{1-\frac{a}{2}+m,1-\frac{a}{2}+m,\frac{1-a}{2}+m,\frac{1-a}{2}+m}{-\frac{4\pi^6n^2}{y^2D_{\mathbb{K}}^2}}\nonumber\\
&-2(-1)^m\Gamma\left(1+\frac{a}{2}-m\right)\Gamma\left(1-\frac{a}{2}+m\right)\left(\frac{4\pi^6n^2}{y^2D_{\mathbb{K}}^2}\right)^{\frac{a}{4}}\mathrm{ber}\left(4\pi\sqrt{\frac{2n\pi}{yD_{\mathbb{K}}}}\right)\Bigg\}.
\end{align*}
Equation \eqref{BerBei} implies that the above equation reduces to $0/0$ form for $a=2m$. Thus we can use L'Hopital's rule to evaluate the limit as
\begin{align*}
L_{2m}&=(-1)^{m+1}2^m\pi^{3m}\left(\frac{n}{yD_{\mathbb{K}}}\right)^m\Bigg\{\mathrm{ber}\left(4\pi\sqrt{\frac{2n\pi}{yD_{\mathbb{K}}}}\right)\left(4\gamma+\log\left(\frac{64\pi^6n^2}{y^2D_{\mathbb{K}}^2}\right)\right)\nonumber\\
&\quad-2\frac{d}{da}\pFq14{1}{1-\frac{a}{2}+m,1-\frac{a}{2}+m,\frac{1-a}{2}+m,\frac{1-a}{2}+m}{-\frac{4\pi^6n^2}{y^2D_{\mathbb{K}}^2}}\Bigg\}\Bigg|_{a=2m}.
\end{align*}
Inserting \eqref{derivative2} with $\ell=m$ into the above equation, we evaluate
\begin{align}\label{conclude3}
L_{2m}&=(-1)^{m+1}2^m\pi^{3m}\left(\frac{n}{yD_{\mathbb{K}}}\right)^m\left\{\pi\mathrm{bei}\left(4\pi\sqrt{\frac{2n\pi}{yD_{\mathbb{K}}}}\right)-4\mathrm{ker}\left(4\pi\sqrt{\frac{2n\pi}{yD_{\mathbb{K}}}}\right)\right\}.
\end{align}
It follows immediately from the poles of the gamma functions that
\begin{align}\label{effortless01}
&\lim_{a\to 2m}\Bigg\{\frac{\Gamma\left(\frac{a}{2}\right)\Gamma\left(\frac{1+a}{2}\right)}{\Gamma\left(1-\frac{a}{2}\right)\Gamma\left(\frac{1-a}{2}\right)}\left(\frac{4\pi^6n^2}{y^2D_{\mathbb{K}}^2}\right)^{-\frac{a}{4}}-\Gamma\left(\frac{a}{2}\right)\Gamma\left(\frac{1+a}{2}\right)\Gamma\left(1-\frac{a}{2}\right)\Gamma\left(\frac{1-a}{2}\right)\left(\frac{4\pi^6n^2}{y^2D_{\mathbb{K}}^2}\right)^{-\frac{a}{4}}\nonumber\\
&\quad\times\sum_{k=0}^{m-2}\frac{(-1)^k}{\Gamma\left(2-\frac{a}{2}+k\right)^2\Gamma\left(\frac{3-a}{2}+k\right)^2}\left(\frac{4\pi^6n^2}{y^2D_{\mathbb{K}}^2}\right)^{k+1}\Bigg\}=0.
\end{align}
Thus taking limit as $a\to2m$ on the both sides of \eqref{Eqn:Lambert series} and using the fact that $\zeta_{\K}(s)$ has zeros on the negative integers, \eqref{expanded form}, \eqref{conclude3} and \eqref{effortless01} together yield
\begin{align}\label{a2m in y form}
\sum_{n=1}^\infty\sigma_{\mathbb{K},2m}(n)e^{-ny}&=\frac{L(1,\chi_D)\Gamma(2m+1)\zeta(2m+1)}{y^{2m+1}}+\frac{4(-1)^m}{\sqrt{D_{\mathbb{K}}}}\left(\frac{2\pi}{y}\right)^{2m+1}\sum_{n=1}^\infty \frac{\sigma_{\mathbb{K},-2m}(n)}{n^{-2m}}\mathrm{Ker}\left(4\pi\sqrt{\frac{2n\pi}{yD_{\mathbb{K}}}}\right).
\end{align}
Finally, we substitute $y$ by $\frac{8\pi^2\alpha}{D_{\K}}$ with $\alpha\beta=\frac{D_{\K}^2}{16\pi^2}$ in \eqref{a2m in y form} to conclude \eqref{a=2m}.

\section{Transformation formulas analogous to Ramanujan's identity for $\zeta(2m+1)$}
In this section, we exhibit an identity over imaginary quadratic field which is analogue to Ramanujan's identity \eqref{Ramanujan formula}. 

\begin{lemma}\label{zetk prime}
For any natural number $n>1$, we have
\begin{align*}
\zeta_{\K}'(1-n)=(-1)^{n-1}D_{\K}^{n-\frac{1}{2}}(2\pi)^{1-2n}((n-1)!)^2\zeta_{\K}(n).
\end{align*}
\end{lemma}
\begin{proof}
The proof follows by taking derivative on the both sides of the functional equation \eqref{zetaK functional equation}.
\end{proof}

\subsection{Proof of Theorem \ref{Ramanujan analogue}}
For the first part, the idea here is to take limit $a\to-2m-1$ on the both sides of \eqref{Eqn:Analytic continuation}. We first evaluate the following limit:
\begin{align*}
L_{-2m-1}:=& \lim_{a\to -2m-1} \frac{1}{a+2m+1} \Bigg\{(-1)^m \frac{2\pi^2}{\sin(\pi a)}\frac{(a+2m+1)\left(\frac{4\pi^6 n^2}{y^2 D_{\K}^2}\right)^{-\frac{a}{4}-m}}{\Gamma^2\left(1-\frac{a}{2}-m\right) \Gamma^2\left(\frac{1-a}{2}-m\right)} \nonumber\\
&\times {}_1F_4\left(\begin{matrix}
1\\ 1 -\frac{a}{2}-m,  1 -\frac{a}{2}-m,  \frac{1-a}{2}-m, \frac{1-a}{2}-m 
\end{matrix} \bigg | -\frac{4\pi^6 n^2}{y^2 D_{\K}^2} \right)\nonumber\\
&-(a+2m+1)\Gamma\left(-\frac{1}{2}-\frac{a}{2}\right)\Gamma\left(\frac{3+a}{2}\right)\left(\frac{4\pi^{6}n^2}{y^2D_{\mathbb{K}}^2}\right)^{\frac{a}{4}}\mathrm{bei}\left(4\pi\sqrt{\frac{2n\pi}{yD_{\mathbb{K}}}}\right)\Bigg\}.
\end{align*}
Invoking \eqref{Reflection formula} on the gamma factors, the above equation yields
\begin{align*}
L_{-2m-1} &=  \lim_{a\to -2m-1} \frac{1}{a+2m+1} \Bigg\{(-1)^{m+1} 2\pi\frac{\Gamma(a+2m+2) \Gamma(-a-2m)}{\Gamma^2\left(1-\frac{a}{2}-m\right) \Gamma^2\left(\frac{1-a}{2}-m\right)} \left(\frac{4\pi^6 n^2}{y^2 D_{\K}^2}\right)^{-\frac{a}{4}-m} \nonumber\\
&\times {}_1F_4\left(\begin{matrix}
1\\ 1 -\frac{a}{2}-m,  1 -\frac{a}{2}-m,  \frac{1-a}{2}-m, \frac{1-a}{2}-m 
\end{matrix} \bigg | -\frac{4\pi^6 n^2}{y^2 D_{\K}^2} \right)\nonumber\\
&-2(-1)^{m+1} \Gamma\left(1+\frac{a+2m+1}{2}\right) \Gamma\left(1-\frac{a+2m+1}{2}\right) \left(\frac{4\pi^{6}n^2}{y^2D_{\mathbb{K}}^2}\right)^{\frac{a}{4}}\mathrm{bei}\left(4\pi\sqrt{\frac{2n\pi}{yD_{\mathbb{K}}}}\right)\Bigg\}.
\end{align*}
We now apply \eqref{BerBei} on the above equation to show that the above limit is of the form $0/0$ for $a=-2m-1$. Thus we can use L'Hopital's rule and apply \eqref{derivative1} with $\ell=-m$ to evaluate the limit as
\begin{align}\label{finalliml-2m-1}
L_{-2m-1} &=(-1)^{m+1}\left(\frac{2\pi^3n}{yD_{\mathbb{K}}}\right)^{-\frac{1}{2}-m}\left\{\pi\mathrm{ber}\left(4\pi\sqrt{\frac{2n\pi}{yD_{\mathbb{K}}}}\right)+4\mathrm{kei}\left(4\pi\sqrt{\frac{2n\pi}{yD_{\mathbb{K}}}}\right)\right\}.
\end{align}
It can also be observed that
\begin{align}\label{a=-2m-1lim0}
\lim_{a\to-2m-1}\left\{\frac{2^{-2a-3}\pi}{\sin (\pi a)}\left(\frac{2\pi^3n}{yD_{\K}}\right)^{-\frac{a}{2}-2}\frac{(-1)^m}{\Gamma^2(-a-1-2m)}\left(\frac{4\pi^6n^2}{y^2D_\K^2}\right)^{-m}\right\}=0.
\end{align}
We next investigate the last term on the right-hand side of \eqref{Eqn:Analytic continuation} which is of the form $0/0$ as $a\to-2m-1$, for $0\leq k\leq m-1$ due to the zeros of $\sin(\pi a)$ and $\zeta_{\mathbb{K}}(a+2k+2)$ but for $k=m$, the term reduces to $\infty/\infty$ form. Thus we evaluate these two limits separately using L' Hopital rule as
\begin{align}\label{finitesumlim2}
\lim_{a\to-2m-1}&\left\{\frac{D_{\mathbb{K}}^{a+\frac{3}{2}}}{(2\pi)^{2a+4}\sin(\pi a)}\sum_{k=0}^{m-1}\frac{(-1)^k\zeta(2k+2)\zeta_{\mathbb{K}}(a+2k+2)}{\Gamma^2(-1-a-2k)}\left(\frac{8\pi^3}{yD_{\mathbb{K}}}\right)^{-2k}\right\}\nonumber\\
&=\frac{D_{\mathbb{K}}^{\frac{1}{2}-2m}}{(2\pi)^{2-4m}}\sum_{k=0}^{m-1}\frac{(-1)^k\zeta(2k+2)}{\Gamma^2(2m-2k)}\lim_{a\to-2m-1}\left\{\frac{\zeta_{\mathbb{K}}'(a+2k+2)}{\cos(\pi a)\pi}\right\}\left(\frac{8\pi^3}{yD_{\mathbb{K}}}\right)^{-2k}\nonumber\\
&=-\frac{D_{\mathbb{K}}^{\frac{1}{2}-2m}}{\pi(2\pi)^{2-4m}}\sum_{k=0}^{m-1}\frac{(-1)^k\zeta(2k+2)\zeta_{\mathbb{K}}'(2k-2m+1)}{\Gamma^2(2m-2k)}\left(\frac{8\pi^3}{yD_{\mathbb{K}}}\right)^{-2k},
\end{align}
and
\begin{align}\label{finitesumlim1}
\lim_{a\to-2m-1}\left\{\frac{\zeta_{\mathbb{K}}(a+2m+2)}{\sin(\pi a)\Gamma^2(-1-a-2m)}\right\}&=\lim_{a\to-2m-1}\left\{\frac{(a+2m+1)\zeta_{\mathbb{K}}(a+2m+2)}{\frac{\sin(\pi a)}{(a+2m+1)}((a+2m+1)^2)\Gamma^2(-1-a-2m)}\right\}\nonumber\\
&=-\frac{2h}{w\sqrt{D_{\K}}}.
\end{align}
Taking limit as $a\to-2m-1$ overall in \eqref{Eqn:Analytic continuation}, the evaluations \eqref{finalliml-2m-1}, \eqref{a=-2m-1lim0}, \eqref{finitesumlim2} and \eqref{finitesumlim1} together yield
\begin{align}\label{ramanalogeqn}
&\sum_{n=1}^\infty\sigma_{\mathbb{K},-2m-1}(n)e^{-ny}=-\frac{1}{2}\zeta_{\mathbb{K}}(2m+1)+\frac{1}{y}\zeta_{\mathbb{K}}(2m+2)+\frac{(-1)^m2^{-2m}h\zeta(2m+1)}{w\sqrt{D_{\K}}\pi^{2m-1}y^{-2m}}\nonumber\\
&-\frac{(-1)^m(2\pi)^{4m-1}h}{\pi w\sqrt{D_{\K}} D_{\mathbb{K}}^{2m-\frac{1}{2}}}\zeta(2m+2)\left(\frac{8\pi^3}{yD_{\mathbb{K}}}\right)^{-2m}+\frac{(-1)^{m+1}2^{2-2m}\pi^{-2m}}{y^{-2m}\sqrt{D_{\mathbb{K}}}}\sum_{n=1}^\infty \frac{\sigma_{\mathbb{K},2m+1}(n)}{n^{2m+1}}\mathrm{kei}\left(4\pi\sqrt{\frac{2n\pi}{yD_{\mathbb{K}}}}\right)\nonumber\\ 
&-\frac{(2\pi)^{4m-2}}{\pi D_{\mathbb{K}}^{2m-\frac{1}{2}}}\sum_{k=0}^{m-1}\frac{(-1)^k\zeta(2k+2)\zeta'_{\mathbb{K}}(2k-2m+1)}{\Gamma^2(2m-2k)}\left(\frac{8\pi^3}{yD_{\mathbb{K}}}\right)^{-2k}.
\end{align}
Finally, we replace $k$ by $m-1-k$ and apply Lemma \ref{zetk prime} in the last term of the above equation then substitute $y$ by $\frac{8\pi^2\alpha}{D_{\K}}$ with $\alpha\beta=\frac{D_{\K}^2}{16\pi^2}$ in \eqref{ramanalogeqn} to arrive at \eqref{a=-2m-1 case}.

We next show the second part of our theorem. The idea of the proof goes along the similar direction as the previous part by taking limit as $a\to -2m$ on the both sides of \eqref{Eqn:Analytic continuation}. In this case, we need to determine the following limit:
\begin{align*}
L_{-2m}:&=\lim_{a\to-2m}\frac{1}{(a+2m)}\Bigg\{\frac{2(-1)^m \pi^2(a+2m)}{\sin(\pi a)\Gamma^2\left(1-\frac{a}{2}-m\right) \Gamma^2\left(\frac{1-a}{2}-m\right)}\left(\frac{4\pi^6 n^2}{y^2 D_{\mathbb{K}}^2}\right)^{-\frac{a}{4}-m} \nonumber\\
&\times {}_1F_4\left(\begin{matrix}
1\\ 1 -\frac{a}{2}-m,  1 -\frac{a}{2}-m,  \frac{1-a}{2}-m, \frac{1-a}{2}-m 
\end{matrix} \bigg | -\frac{4\pi^6 n^2}{y^2 D_{\mathbb{K}}^2} \right)\nonumber\\
&+(a+2m)\Gamma\left(-\frac{a}{2}\right)\Gamma\left(1+\frac{a}{2}\right)\left(\frac{4\pi^{6}n^2}{y^2D_{\mathbb{K}}^2}\right)^{\frac{a}{4}}\mathrm{ber}\left(4\pi\sqrt{\frac{2n\pi}{yD_{\mathbb{K}}}}\right)\Bigg \}.
\end{align*}
Applying \eqref{Reflection formula} on the gamma factors of the above equation, we obtain
\begin{align*}
L_{-2m}&=\lim_{a\to-2m}\frac{1}{(a+2m)}\Bigg\{ \frac{2\pi(-1)^m\Gamma(a+2m+1)\Gamma(1-a-2m)}{\Gamma^2\left(1-\frac{a}{2}-m\right) \Gamma^2\left(\frac{1-a}{2}-m\right)}\left(\frac{4\pi^6 n^2}{y^2 D_{\mathbb{K}}^2}\right)^{-\frac{a}{4}-m} \nonumber\\
&\times {}_1F_4\left(\begin{matrix}
1\\ 1 -\frac{a}{2}-m,  1 -\frac{a}{2}-m,  \frac{1-a}{2}-m, \frac{1-a}{2}-m 
\end{matrix} \bigg | -\frac{4\pi^6 n^2}{y^2 D_{\mathbb{K}}^2} \right)\nonumber\\
&-2(-1)^m\Gamma\left(1+\frac{a+2m}{2}\right)\Gamma\left(1-\frac{a+2m}{2}\right)\left(\frac{4\pi^{6}n^2}{y^2D_{\mathbb{K}}^2}\right)^{\frac{a}{4}}\mathrm{ber}\left(4\pi\sqrt{\frac{2n\pi}{yD_{\mathbb{K}}}}\right)\Bigg \}.
\end{align*}
It is clear by \eqref{BerBei} that the above limit reduces to $0/0$ form. Thus,  L'Hopital's rule is applicable to evaluate the limit. Applying it and using \eqref{derivative2} with $\ell=-m$ after simplification on the above limit, we have
\begin{align}\label{vl-2m3}
L_{-2m}&=(-1)^m2^{-m}\pi^{-3m}\left(\frac{yD_{\mathbb{K}}}{n}\right)^m\left\{4\mathrm{ker}\left(4\pi\sqrt{\frac{2n\pi}{yD_{\mathbb{K}}}}\right)-\pi\mathrm{bei}\left(4\pi\sqrt{\frac{2n\pi}{yD_{\mathbb{K}}}}\right)\right\}.
\end{align}
It is easy to see that
\begin{align}\label{vl-2mzero}
\lim_{a\to-2m}\left\{\frac{2^{-2a-3}\pi}{\sin (\pi a)}\left(\frac{2\pi^3n}{yD_{\mathbb{K}}}\right)^{-\frac{a}{2}-2}\frac{(-1)^m}{\Gamma^2(-a-1-2m)}\left(\frac{4\pi^6n^2}{y^2D_\mathbb{K}^2}\right)^{-m}\right\}=0.
\end{align}
Next, we evaluate the finite sum on the right-hand side of \eqref{Eqn:Analytic continuation} as $a\to-2m$, which is $0/0$ form for $0\leq k\leq m-2$ due to the zeros of $\sin(\pi a)$ and $\zeta_{\mathbb{K}}(a+2k+2)$. The $m$-th term of the finite sum goes to zero as $s\to-2m$ because of the double pole of $\Gamma^2(-1-a-2k)$ in the denominator. Next we show that the addition of $(m-1)$-th term of the finite sum and the fourth term on the left-hand side of \eqref{Eqn:Analytic continuation} provides $0/0$ form and for that we use the functional equation of $\zeta(s)$ in the asymmetric form to obtain
\begin{align*}
&\frac{(-1)^{m-1}yD_{\mathbb{K}}^{a+\frac{3}{2}}\zeta(2m)\zeta_{\mathbb{K}}(2m+a)}{(2\pi)^{2a+4}\sin(\pi a)\Gamma^2(1-a-2m)}\left(\frac{64\pi^6}{y^2D_{\mathbb{K}}^2}\right)^{1-m}-\frac{2\pi h\Gamma(a+1)\zeta(a+1)}{w\sqrt{D_{\K}} y^{a+1}}\nonumber\\
&=\frac{1}{\sin\left(\frac{\pi a}{2}\right)}\left\{\frac{(-1)^{m-1}yD_{\mathbb{K}}^{a+\frac{3}{2}}\zeta(2m)\zeta_{\mathbb{K}}(2m+a)}{2(2\pi)^{2a+4}\cos\left(\frac{\pi a}{2}\right)\Gamma^2(1-a-2m)}\left(\frac{64\pi^6}{y^2D_{\mathbb{K}}^2}\right)^{1-m}-\frac{2^{a+1}\pi^{a+2}h\zeta(-a)}{w\sqrt{D_{\K}}y^{a+1}}\right\}.
\end{align*}
The fact $\zeta_{\mathbb{K}}(0)=-\frac{h}{w}$ exhibits that the term inside the bracket on the right-hand side of the above expression is $0$ for $a=-2m$. Thus we have $0/0$ form on the above limit where we can apply  L'Hopital's rule to evaluate the limit as
\begin{align*}
&\lim_{a\to-2m}\left\{\frac{(-1)^{m-1}yD_{\mathbb{K}}^{a+\frac{3}{2}}\zeta(2m)\zeta_{\mathbb{K}}(2m+a)}{(2\pi)^{2a+4}\sin(\pi a)\Gamma^2(1-a-2m)}\left(\frac{64\pi^6}{y^2D_{\mathbb{K}}^2}\right)^{1-m}-\frac{2\pi h\Gamma(a+1)\zeta(a+1)}{w\sqrt{D_{\K}}y^{a+1}}\right\}\nonumber\\
&=\frac{2(-1)^m}{\pi}\Bigg\{\frac{(-1)^{m-1}yD_{\mathbb{K}}^{\frac{3}{2}}\zeta(2m)}{2(2\pi)^4}\left(\frac{64\pi^6}{y^2D_{\mathbb{K}}^2}\right)^{1-m}\Bigg[(-1)^m\zeta_{\mathbb{K}}(0)\left(\frac{D_\mathbb{K}}{4\pi^2}\right)^{-2m}\left(\log\left(\frac{D_\mathbb{K}}{4\pi^2}\right)-2\gamma\right)\nonumber\\
&\quad+(-1)^m\zeta'_{\mathbb{K}}(0)\left(\frac{D_\mathbb{K}}{4\pi^2}\right)^{-2m}\Bigg]+\frac{2\pi^2h}{yw\sqrt{D_{\K}}}\left(\frac{2\pi}{y}\right)^{-2m}\left(\zeta'(2m)-\zeta(2m)\log\left(\frac{2\pi}{y}\right)\right)\Bigg\}.
\end{align*}
The facts $\zeta_{\mathbb{K}}(0)=-\frac{\sqrt{D_{\mathbb{K}}}}{2\pi}L(1,\chi_{D}),\ \zeta'_{\mathbb{K}}(0)=\frac{\sqrt{D_{\mathbb{K}}}}{2\pi}L'(1,\chi_{D})-\frac{\sqrt{D_{\mathbb{K}}}}{2\pi}\left(\gamma-2\log\left(\frac{\sqrt{D_{\mathbb{K}}}}{2\pi}\right)\right)L(1,\chi_{D})$ and Proposition \ref{Class number formula} reduce the above equation as 
\begin{align}\label{limitofm-1andthirdterm}
&\lim_{a\to-2m}\left\{\frac{(-1)^{m-1}yD_{\mathbb{K}}^{a+\frac{3}{2}}\zeta(2m)\zeta_{\mathbb{K}}(2m+a)}{(2\pi)^{2a+4}\sin(\pi a)\Gamma^2(1-a-2m)}\left(\frac{64\pi^6}{y^2D_{\mathbb{K}}^2}\right)^{1-m}-\frac{2\pi h\Gamma(a+1)\zeta(a+1)}{w\sqrt{D_\K} y^{a+1}}\right\}\nonumber\\
&=\frac{(-1)^m}{\pi}\left(\frac{y}{2\pi}\right)^{2m-1}\left\{\left(\zeta'(2m)-\gamma \zeta(2m)-\zeta(2m)\log\left(\frac{2\pi}{y}\right)\right)L(1,\chi_{D})-L'(1,\chi_{D})\zeta(2m)\right\}.
\end{align}
We also have
\begin{align}\label{a=-2mfinitesumlimit}
&\lim_{a\to-2m}\left\{\frac{D_{\mathbb{K}}^{a+\frac{3}{2}}y}{(2\pi)^{2a+4}\sin(\pi a)}\sum_{k=0}^{m-2}(-1)^k\frac{\zeta(2k+2)\zeta_{\mathbb{K}}(a+2k+2)}{\Gamma^2(-1-a-2k)}\left(\frac{8\pi^3}{yD_{\mathbb{K}}}\right)^{-2k}\right\}\nonumber\\
&=\frac{yD_{\mathbb{K}}^{\frac{3}{2}-2m}}{\pi(2\pi)^{4-4m}}\sum_{k=0}^{m-2}(-1)^k\frac{\zeta(2k+2)\zeta'_{\mathbb{K}}(2k-2m+2)}{\Gamma^2(2m-2k-1)}\left(\frac{64\pi^6}{y^2D_{\mathbb{K}}^2}\right)^{-k}.
\end{align}
Invoking \eqref{vl-2m3}, \eqref{vl-2mzero}, \eqref{limitofm-1andthirdterm} and \eqref{a=-2mfinitesumlimit} and taking limit as $a\to-2m$ in Theorem \ref{Eqn:Analytic continuation}, we arrive at
\begin{align}\label{casea=-2meqn}
\sum_{n=1}^\infty\sigma_{\mathbb{K},-2m}(n)&e^{-ny} =-\frac{1}{2}\zeta_{\mathbb{K}}(2m)+\frac{1}{y}\zeta_{\mathbb{K}}(2m+1)+\frac{4(-1)^m}{\pi\sqrt{D_{\mathbb{K}}}}\left(\frac{2\pi^2}{y}\right)^{1-2m}\sum_{n=1}^\infty\frac{\sigma_{\mathbb{K},2m}(n)}{n^{2m}}\mathrm{ker}\left(4\pi\sqrt{\frac{2n\pi}{yD_{\mathbb{K}}}}\right)\nonumber\\
&+\frac{(-1)^m}{\pi}\left(\frac{y}{2\pi}\right)^{2m-1}\left\{\left(\zeta'(2m)-\gamma\zeta(2m)-\zeta(2m)\log\left(\frac{2\pi}{y}\right)\right)L(1,\chi_D)-L'(1,\chi_D)\zeta(2m)\right\}\nonumber\\
&+\frac{y}{2}\left(\frac{2\pi}{\sqrt{D_\K}}\right)^{4m-3} \, \ \sum_{k=0}^{m-2}\frac{(-1)^k\zeta(2k+2)\zeta'_{\mathbb{K}}(2k-2m+2)}{\Gamma^2(2m-2k-1)}\left(\frac{64\pi^6}{y^2D_{\mathbb{K}}^2}\right)^{-k}.
\end{align}
Finally, we replace $k$ by $m-1-k$ and apply Lemma \ref{zetk prime} in the last term of the above equation then substitute $y$ by $\frac{8\pi^2\alpha}{D_{\K}}$ with $\alpha\beta=\frac{D_{\K}^2}{16\pi^2}$ in \eqref{casea=-2meqn} to conclude \eqref{a=-2m case}. This completes the proof of Theorem \ref{Ramanujan analogue}.

\section{Concluding remarks}
Zagier \cite{Zagier} asked whether there is a formula for $\zeta_\K(4), \zeta_\K(6)$ etc. attached an arbitrary imaginary quadratic field $\K$ similar to \eqref{Eqn: Zagier}. He also remarked that the answer to this question is possible if one can prove the result using methods available in analytic number theory. Here, we have obtained the relation between two zeta values for any complex arguments in terms of infinite series. Moreover, Theorem \ref{Ramanujan analogue} provides an explicit relation between even and odd zeta values over an imaginary quadratic field. Thus the expression of $\zeta_\K(2)$ (cf. Theorem \ref{a=-1 case}) together with Theorem \ref{Ramanujan analogue} expresses any zeta value at positive integers over an imaginary quadratic field in terms of Lambert series :
\begin{align*}
\sum_{n=1}^\infty\sigma_{\mathbb{K}, a}(n)e^{-ny}, \hspace{1cm} \sum_{n=1}^\infty\sigma_{\mathbb{K}, a}(n)\mathrm{ker}(\sqrt{ny})  \hspace{.5cm} \text{and} \hspace{.5cm} \sum_{n=1}^\infty\sigma_{\mathbb{K}, a}(n)\mathrm{kei}(\sqrt{ny}).
\end{align*}
These series demand independent study since their behaviour may lead to some important information about the arithmetic nature of Dedekind zeta function over an imaginary quadratic field.

\vspace{.4cm}
\noindent
\textbf{Acknowledgements.}
The authors would like to show their sincere gratitude to Prof. Atul Dixit for the fruitful discussions and suggestions on the manuscript. The authors are also grateful to IIT Gandhinagar and Prof. Atul Dixit for their financial support.


\begin{thebibliography}{00}

\bibitem{AAR} G.~E.~Andrews, R.~Askey and R.~Roy, \emph{Special Functions}, Encyclopedia of Mathematics and its Applications, 71, Cambridge University Press,  Cambridge, 1999.

\bibitem{Apery} R.~Apery, {\em Irrationalit{\'e} de $\zeta(2)$ et $\zeta(3)$}, Ast{\'e}risque {\bf 61} (1979), 11--13.

\bibitem{Ayoub} R.~G.~Ayoub, {\em An introduction to the analytic theory of numbers}, American Math. Soc., (1963).

\bibitem{Banerjee} S.~Banerjee, {\em Piltz divisor problem over number fields {\`a} la Vorono\"{\dotlessi}}, Proc. Amer. Math. Soc., {\bf 149} (2021), 1025--1038.

\bibitem{BCH} S.~Banerjee, K.~Chakraborty and A.~Hoque, {\em An analogue of Wilton's formula and values of Dedekind zeta functions},  J. Math. Anal. Appl., {\bf 495} (2021), 124675, 20 pp.


\bibitem{BLS} B.~C.~Berndt, Y.~Lee, and J.~Sohn, {\em The formulas of Koshliakov and Guinand in Ramanujan’s lost notebook}, Surveys in Number Theory, Series: Developments in Mathematics, Vol. {\bf 17}, K. Alladi, ed., Springer, New York, 2008, pp. 21--42.

\bibitem{BS} B.~C.~Berndt and A.~Straub, {\em Ramanujan’s formula for $\zeta(2n + 1)$, Exploring the Riemann zeta function}, 13--34, Springer, Cham, 2017.



\bibitem{bdrz} B.~C.~Berndt, A.~Dixit, A.~Roy and A.~Zaharescu, \emph{New pathways and connections in number theory and analysis motivated by two incorrect claims of Ramanujan}, Advances in Mathematics {\bf 304} (2017), 809--929.

\bibitem{Bordelles} O.~Bordelles,  {\em Arithmetic tales}, London: Springer, 2012.

\bibitem{bry} A.~Y.~Brychkov, {Handbook of special functions: derivatives, integrals, series and other formulas} CRC press, 2008.



\bibitem{DKK} A.~Dixit, A.~Kesarwani and R.~Kumar, {\em A generalized modified Bessel function and explicit transformations of certain Lambert series}, submitted for publication. 





\bibitem{Huxley} M.~N.~Huxley, {\em Exponential sums and lattice points}, III, Proc. Lond. Math. Soc. (3) {\bf 87} (2003), 591--609.

\bibitem{Iwaniec} H.~Iwaniec and E.~Kowalski, {\em Analytic Number Theory}, Amer. Math. Soc. Colloquium Publ. {\bf 53}, Amer. Math. Soc., Providence, RI, (2004).

\bibitem{Klingen} H.~Klingen, {\em {\"U}ber die Werte der DedekindschenZetafunktionen}, Math. Ann. {\bf 145} (1962), 265--272.

\bibitem{Koshliakov} N.~S.~Koshliakov, {\em On Vorono\"{\dotlessi}'s sum-formula}, Mess. Math. {\bf 58} (1929), 30--32.


\bibitem{Lang} S.~Lang, {\em Algebraic number theory}, Addison-Wesley: Reading, MA, (1970).


\bibitem{lerch}
M.~Lerch, \emph{Sur la fonction $\zeta(s)$ pour valeurs impaires de l’argument}, J. Sci. Math. Astron. pub. pelo Dr. F. Gomes Teixeira, Coimbra 14 (1901), 65--69.

\bibitem{Luke}
Y.~L.~Luke,  \emph{Special Functions and Their Approximations} v. {\bf 1}, Academic press, (1969). 

\bibitem{Marichev} O.~I.~Marichev, {\it Handbook of Integral Transforms of Higher Transcendental Functions: Theory and Algorithmic Tables}, The Ellis Horwood Series in Mathematics and its Applications: Statistics and Operational Research, E. Horwood, 1983.

\bibitem{NIST} F.~W.~J.~Olver, D.~W.~Lozier, R.~F.~Boisvert and C.~W.~Clark, eds., {\em NIST Handbook of Mathematical Functions}, Cambridge University Press, Cambridge, 2010.

\bibitem{Paris} R.~B.~Paris and D.~Kaminski, {\em Asymptotics and Mellin-Barnes Integrals}, Encyclopedia of Mathematics and its Applications, {\bf 85}, Cambridge University Press, Cambridge, 2001.

\bibitem{Prudnikov} A.~P.~Prudnikov, Yu.~A.~Brychkov and O.~I.~Marichev, {\em Integrals and series}, Vol. 1: Elementary functions,
Gordan and Breach, New York, 1986.

\bibitem{Ramanujan} S.~Ramanujan, {\em The Lost Notebook and Other Unpublished Papers}, Narosa, New Delhi, 1988.


\bibitem{Siegel} C.~L.~Siegel, {\em Berechnung von Zetafunktionen an ganzzahligen Stellen}, Nachr. Akad. Wiss. G{\"o}tt., {\bf 10}, 87--102, 1969.


\bibitem{Voronoi} G.~Vorono\"{\dotlessi}, {\em Sur une fonction transcendante et ses applications \`a la sommation de quelques s\'eries}, Ann. Ecole Norm. Sup., {\bf 21} (1904), 207--267, 459--533.

\bibitem{WB} N.~L.~Wang and S.~Banerjee, {\em On the product of Hurwitz zeta-functions}, Proc. Japan Acad., Ser. A, Math. Sci. {\bf 93 : 5} (2017), 31--36.

\bibitem{watson-1944a}
G.~N.~Watson, \emph{A Treatise on the Theory of Bessel Functions}, second ed., Cambridge University Press, London, 1944.

\bibitem{Zagier} D.~Zagier, {\em Hyperbolic manifolds and special values of Dedekind zeta-functions}, Invent. Math. {\bf 83} (1986), 285–301.

\bibitem{Zagier2} D.~Zagier, {\em Values of zeta functions and their applications}, First European Congress of Mathematics Paris, July 6–10 (1992), 497--512. Birkhäuser Basel, 1994.

\bibitem{Zudilin} V.~Zudilin, {\em One of the numbers $\zeta(5)$, $\zeta(7)$, $\zeta(9)$, $\zeta(5)$ is irrational}, (Russian) Uspekhi Mat. Nauk {\bf 56} (2001), no. 4(340), 149--150; translation in Russian Math. Surveys {\bf 56} (2001), no. 4, 774--776.

\end{thebibliography}
\end{document}